\documentclass[10pt]{article}

\usepackage{amsmath, amsfonts, amssymb, amsthm, amstext}
\usepackage{enumitem}
\usepackage[calcwidth]{titlesec}
\usepackage{graphicx}
\usepackage{subfigure}
\usepackage{appendix}
\usepackage[colorlinks=true,linkcolor=blue,citecolor=blue,pagebackref]{hyperref}
\usepackage{cleveref}
\usepackage{tocbibind}

\theoremstyle{theorem}
\newtheorem{theorem}{Theorem}
\newtheorem{lemma}[theorem]{Lemma}

\newcommand{\Ld}{\mathbf{L^2}}
\newcommand{\Lu}{\mathbf{L^1}}
\newcommand{\om}{{\omega}}
\newcommand{\C}{\mathbb{C}}
\newcommand{\Ehk}{E_{\mathrm{hk}}}

\newcommand{\N}{\mathbb{N}}
\newcommand{\OO}{\mathrm{O}}
\newcommand{\R}{\mathbb{R}}
\newcommand{\SO}{\mathrm{SO}}
\newcommand{\Z}{\mathbb{Z}}
\newcommand{\argmax}{\mathrm{arg}\max}
\newcommand{\argmin}{\mathrm{arg}\min}
\newcommand{\supp}{\mathrm{supp}}
\newcommand{\hrho}{\hat{\rho}}
\newcommand{\nrho}{\rho_{\mathrm{n}}}
\newcommand{\valrho}{\rho_{\mathrm{val}}}
\newcommand{\corrho}{\rho_{\mathrm{cor}}}
\newcommand{\hpsi}{\widehat{\psi}}
\newcommand{\hphi}{\widehat{\phi}}
\newcommand{\lb}{{\langle}}
\newcommand{\rb}{{\rangle}}

\title{Wavelet Scattering Regression of Quantum Chemical Energies%
  \thanks{Submitted to \textit{Multiscale Modeling and Simulation} May 16, 2016. Revised version
    submitted October 24, 2016. Accepted for publication in
    \textit{Multiscale Modeling and Simulation} January 3, 2017. \newline
    \textbf{Funding:} The research and manuscript preparation were
    supported by ERC grant InvariantClass 320959. During the revision of the
    manuscript, M.H. was partially supported by an Alfred P. Sloan
    Research Fellowship, a DARPA Young Faculty Award, NSF grant
    number 1620216, and additionally the Institute for Pure and Applied
    Mathematics (IPAM) during which time he was a resident.}}

\author{Matthew Hirn%
\thanks{Michigan State University, Department of Computational
  Mathematics, Science \& Engineering and Department of Mathematics,
  428 South Shaw Lane, East Lansing, MI, 48824, USA,
  \texttt{mhirn@msu.edu} (corresponding author)}
\and
St\'{e}phane Mallat%
\thanks{\'{E}cole normale sup\'{e}rieure, D\'{e}partement
  d'Informatique, 45 rue d'Ulm, 75005 Paris, France}
\and
Nicolas Poilvert%
\thanks{The Pennsylvania State University, Millennium Science
  Complex, University Park, PA 16801, USA and BayLabs, Inc.,
  \texttt{nicolas@baylabs.io}}}

\usepackage{fancyhdr}
\begin{document}

\maketitle

\begin{abstract}
We introduce multiscale invariant dictionaries to estimate
quantum chemical energies of organic molecules, from training databases.
Molecular energies are invariant to isometric atomic displacements,
and are Lipschitz continuous to molecular deformations.
Similarly to density functional theory (DFT), the molecule is represented
by an electronic density function. A multiscale invariant dictionary
is calculated with wavelet scattering invariants. It 
cascades a first wavelet transform which separates scales, with 
a second wavelet transform
which computes interactions across scales. Sparse scattering regressions
give state of the art results
over two databases of organic planar molecules. On these databases,
the regression error is of the order of the error produced by DFT
codes, but at a fraction of the computational cost. 
\end{abstract}

\newpage

\tableofcontents

\newpage

\section{Introduction}

Computing the energy of a molecule given the charges and relative
positions of its nuclei is a central topic in computational chemistry.
It has important industrial applications such as molecular structure screening for improved organic photovoltaic materials  \cite{Hachmann:CleanEnergy2011} and predicting the thermodynamics and kinetics of industrially relevant chemical reactions \cite{Delgmann:ChemicalIndustry2014}. Molecular energies are low lying eigenvalues
of the molecular Hamiltonian operator, for the system of all
interacting particles in the molecule. An exact numerical computation
of molecular energies is beyond the capabilities of any computer for
all but a handful of tiny molecules. Density functional theory
\cite{staroverov:dft2013} reduces the computation of ground state molecular energies by mapping the eigenvalue problem for the many-electron wavefunction to a variational problem over the total electronic density \cite{hohenberg:dft1964}. Nevertheless,
the computational complexity remains considerable for large molecules. 
\textcolor{black}{Much faster machine learning algorithms 
have been pioneered by several research
groups, with surprisingly good results. Such algorithms for potential
energy surface fitting of one
molecule or material have been developed for more than two decades,
tracing back to \cite{MATS:MATS040030207} with further
contributions in \cite{doi:10.1021/jp055253z, Behler:NNPotEnergy2007,
  :/content/aip/journal/jcp/134/7/10.1063/1.3553717,
  :/content/aip/journal/jcp/139/18/10.1063/1.4828704,
  bartok:repChemEnviron2013,
  :/content/aip/journal/jcp/144/3/10.1063/1.4940026}. More 
recently machine learning
algorithms have been developed to interpolate molecular energies
across ``chemical compound space,'' fitting the potential energy
surface across a range of molecules with different compositions
\cite{rupp:coulombMatrix2012, vonlilienfeld:fourierQC2015, C6CP00415F}. In both
cases molecular energies are regressed through interpolation over a
database of known molecular energies.} The
precision of such algorithms depends upon the
ability to reduce the dimensionality of the regression problem, by focusing
on relevant variables. Efficient regression procedures take advantage of 
existing invariants \cite{bartok:repChemEnviron2013}.
We introduce a sparse multiscale wavelet scattering regression derived from
known invariant and stability properties of molecular energies. 
State of the art numerical results are shown on databases of
planar molecules. 

Most computational chemistry approaches like density functional theory
make use of the Born-Oppenheimer approximation, 
which models the atomic nuclei as classical particles, whereas electrons are
considered as quantum particles. The state 
$x = \{r_k,z_k \}_{k}$ of a molecule is thus 
defined by the position $r_k \in \R^3$ of each nuclei and
its charge $z_k > 0$. \Cref{DFT} briefly
reviews the principles of density functional theory computations.
A regression algorithm uses 
a training set of $n$ molecular states $\{x_i\}_{i\leq n}$ and their
quantum energies $\{f(x_i) \}_{i \leq n}$ to 
approximate the quantum energy $f(x)$ of any molecule
$x$ within a given set $\Omega$. Such a regression $\tilde f(x)$ 
can be calculated
as a linear expansion over a
dictionary  $\Phi(x) = ( \phi_k (x) )_k $ of functions of
the molecular state $x \in \Omega$,
\begin{equation} \label{eqn: linear regression}
\tilde{f}(x) = \lb w , \Phi(x) \rb = \sum_k w_k \,\phi_k(x).
\end{equation}
The regression vector $w = ( w_k )_k$ is computed by
minimizing the training error
$\sum_{i \leq n} |f(x_i) - \tilde f(x_i) |^2$, with some regularity condition imposed
on its norm. Kernel and sparse regressions 
are reviewed in \cref{Energyregress}.

The main difficulty is to find a dictionary $\Phi(x)$ which produces
a small average error $|f(x) - \tilde f(x)|$ over all $x \in \Omega$. 
Since $w$ is calculated from $n$ training values,
$\tilde f$ belongs to an approximation space, not necessarily linear, 
of dimension smaller than the number $n$ of training samples.
The choice of $\Phi(x) = ( \phi_k (x) )_k $ specifies the regularity
of $\tilde f(x)$, which needs to match the regularity of $f(x)$.

A quantum energy functional $f(x)$ satisfies elementary 
invariance and continuity properties relatively to geometric transformations
of $x$, that we summarize:
\begin{enumerate}[topsep=5pt, itemsep=5pt]

\item\label{item:permutation}
{\bf Permutation invariance}: $f(x)$ is invariant to 
permutations of atom indices $k$ in the state vector 
$x = \{r_k,z_k \}_{k}$. 

\item\label{item:isometry}
{\bf Isometric invariance}: $f(x)$ is invariant to global translations,
rotations, and symmetries of atomic positions $r_k$, and hence to any isometry
\cite{bartok:repChemEnviron2013}.

\item\label{item:deformation}
{\bf Deformation stability}: $f(x)$ is Lipschitz continuous 
to variations of the distances $|r_k - r_l|$ between atoms, and hence
Lipschitz continuous to diffeomorphism actions \cite{bartok:repChemEnviron2013}. 
\end{enumerate}
\textcolor{black}{We remark that the third property, deformation stability, excludes
systems in which the potential energy surface abruptly
changes. This phenomena occurs when electronic states cross, as can
happen when a system goes from the ground state to an excited state
after the absorption of a photon. However,
most molecular systems left in their natural state will exhibit wide
temperature and pressure ranges over which the deformation stability
hypothesis is verified. We also emphasize that deformation stability,
as described, does not include the nuclear charges $\{ z_k
\}_k$. These are considered discrete and we only consider the case of
neutral molecular systems.}

Dictionaries invariant to isometries are usually computed from the 
matrix of Euclidean atomic distances  $\{ |r_k - r_l| \}_{k,l}$. 
State of the art
chemical energy regressions have been obtained with Coulomb kernels
computed from these distances. However, \cref{sec: linear regression
  with coulomb kernels} shows that this representation is not
invariant to permutations of the atom indices $k$, which introduces instabilities.

The goal of this paper is to introduce a dictionary $\Phi(x)$ which
is invariant to permutations and isometries, 
and Lipschitz continuous to diffeomorphism actions. It must also 
generate a space which is sufficiently large to provide
accurate sparse energy regressions.
Similarly to density functional theory,
\cref{sec:density} explains how to represent the molecular
state $x$ by an electronic density function
$\rho[x](u)$ of the spatial variable $u \in \R^3$.
This density is invariant to atom index permutations and covariant
to isometries. Invariance to isometries is then obtained
by applying an invariant operator $\Theta$:
\begin{equation*}
\Phi(x) = \Theta \rho [x].
\end{equation*}
The main difficulty is to define an operator $\Theta$ which is
also Lipschitz continuous to diffeomorphisms, 
and whose range is sufficiently large to regress large training databases.
A tempting solution, used by several
quantum energy regression algorithms
\cite{bartok:gaussAppPot2010,bartok:repChemEnviron2013,bartok:GAP2015,
  vonlilienfeld:fourierQC2015, C6CP00415F}, is to
define $\Theta$ from the autocorrelation or the Fourier 
transform modulus of $\rho$. \Cref{Fourier} explains why
it yields instabilities to deformations, which degrades numerical regressions.

Building invariant dictionaries which are stable to diffeomorphism actions
requires to separate variabilities at different
scales. A large body of work has demonstrated the
efficiency of multiscale modeling in quantum chemical systems 
\cite{Natarajan:2012, Walker:BigDFT2016}. \textcolor{black}{Multiscale
  wavelet transforms are used in some density functional theory softwares,
to reduce computations by
providing sparse representations of atomic wavefunctions
\cite{RevModPhys.71.267, nagy:waveletDensity2005,
  Genovese:Daubechies2008, staroverov:dft2013,
  Pei:Wavelet3dHFBNuclear2014, Yanai:2015gb}}. \Cref{sec: wavelet transform} defines wavelet norms 
which are invariant to isometries.
\Cref{sec:classical physics} proves that
Coulomb potential energies are regressed with much fewer
wavelet invariants than Fourier invariants. Such wavelet expansions
are similar to multipole decompositions
\cite{greengard:multipole1987,greengard:multipoleThesis1988}. Quantum
molecular energies also include exchange-correlation
energy terms, responsible for the existence of chemical bonds, which
are more complex than Coulomb potentials.
The numerical results of \cref{sec: numerical results} show that
wavelet invariants do not generate a sufficiently large 
approximation space to precisely regress these exchange-correlation 
energies.

Wavelet invariants perform a multiscale separation but 
do not take into account interactions across scales. 
Scattering invariants have been 
introduced in \cite{mallat:scattering2012} to characterize these
interactions with multiscale interference terms.
They have found applications for image \cite{mallat:rotoScat2013,
  bruna:invariantScatConvNet2013} and audio
\cite{bruna:audioTextureSynth2013, anden:deepScatSpectrum2014} signal
classifications. \Cref{sec: scattering rep} describes the resulting dictionary
computed by transforming wavelet coefficient amplitudes, with a second
wavelet transform. The resulting wavelet scattering dictionary
remains invariant to isometries and Lipschitz continuous to diffeomorphism
actions.

The accuracy of quantum energy regression with
Coulomb kernels, Fourier, wavelet and scattering dictionaries
is numerically evaluated in \cref{sec: numerical results},
over databases of planar molecules. The planar symmetry 
reduces computations to two dimensions.
It is shown that Coulomb kernel regression is approximately
two times more accurate than Fourier and
wavelet regressions. Second order scattering coefficients
considerably reduce the error of wavelet regressions.
It yields sparse expansions of quantum energy functionals. On these
databases, the error is up to twice smaller than Coulomb kernel
regressions and has an accuracy comparable to DFT numerical codes at a
fraction of the computational cost.

\section{Density Functional Theory}
\label{DFT}

Density functional theory (DFT) provides 
relatively fast numerical schemes to compute
the ground state energy of molecular systems. We summarize its
methodology. Following the Born-Oppenheimer approximation, DFT models
the nucleus of each atom as a classical particle, whereas electrons are
considered as quantum particles. 
The total energy of a molecular state $x = \{r_k , z_k \}_k$ is
decomposed into:
\begin{equation*}
\mathcal{E}(x) = E(x) + \frac{1}{2} \sum_{k \neq l}^K \frac{z_k z_l}{|r_k-r_l|}.
\end{equation*}
The second right hand side term is the Coulomb repulsion energy
of all nucleus-nucleus interactions, considered as classical point-wise particles.
The first term $E(x)$ is the quantum energy resulting from 
electronic interactions. It is computed as an eigenvalue of
the electronic Sch\"{o}dinger equation:
\begin{equation} \label{eqn: schrodinger}
H(x) \Psi = E(x) \Psi,
\end{equation}
where $H(x)$ is the Hamiltonian of the state $x$ and 
$\Psi (\overline{r}_1,...,\overline{r}_N)$ is a wavefunction which
depends upon the positions of the $N$ electrons of the molecule. 
The Hamiltonian is a sum of three terms:
\begin{equation*}
H(x) = -\frac{1}{2} \sum_{j=1}^N \Delta_{\overline{r}_j} -
\sum_{j=1}^N \sum_{k=1}^K \frac{z_k}{|\overline{r}_j-r_k|} +
\frac{1}{2} \sum_{i \neq j}^N\frac{1}{|\overline{r}_i - \overline{r}_j|}.
\end{equation*}
These are respectively the electronic kinetic energy, the
Coulomb electron-nucleus attraction energy, and the Coulomb
electron-electron repulsion energy.

The Schr\"{o}dinger equation \eqref{eqn: schrodinger} has many eigenvalues
corresponding to quantized energies.
The ground state energy corresponds to the lowest lying eigenvalue of a molecular state $x$,
that we designate by $E_0 (x)$:
\begin{equation} \label{groundstaten}
f(x) = E_0(x) + \frac{1}{2} \sum_{k \neq l}^K \frac{z_k z_l}{|r_k-r_l|}.
\end{equation}

Computing the Schr\"{o}dinger eigenvalues requires a considerable amount
of computations because $\Psi (\overline{r}_1,...,\overline{r}_N)$ is 
a function of $3N$ variables. Density functional theory considers instead the
total electronic density as its unknown:
\begin{equation*}
\rho (u) = \int_{\R^{3(N-1)}} | \Psi (u, \overline{r}_2,
\ldots, \overline{r}_N) |^2 \, d\overline{r}_2 \cdots d\overline{r}_N.
\end{equation*}
Since we consider the molecule to be electrically neutral, the sum of electronic charges is equal
to the sum of nuclear charges:
\begin{equation*}
\int_{\R^3} \rho (u)\, du = N = \sum_k z_k.
\end{equation*}

Hohenberg and Kohn \cite{hohenberg:dft1964} proved that there exists
a functional $\Ehk(\rho)$ of the electronic density, whose
minimum is the ground state energy $E_0 (x)$. In the Kohn-Sham approach to density
functional theory \cite{PhysRev.140.A1133}, 
the total energy of the system is expressed as such a functional
of the charge density:
\begin{equation} \label{eqn: exact energy dft}
\begin{aligned}
\Ehk(\rho) = \underbrace{T(\rho)}_{\substack{\text{Kinetic energy}}} -
\underbrace{\sum_{k=1}^K \int_{\R^3} \frac{z_k \rho(u)}{|u-r_k|} \,
  du}_{\substack{\text{External energy}\\\text{(electron-nucleus)}}} +
\underbrace{\frac{1}{2} \int_{\R^3} \int_{\R^3} \frac{\rho(u)
    \rho(v)}{|u-v|} \, du \, dv}_{\substack{\text{Hartree energy}\\\text{(electron-electron)}}} +
\underbrace{E_{\mathrm{xc}}(\rho)}_{\substack{\text{Exchange}\\\text{correlation}\\\text{energy}}}.
\end{aligned}
\end{equation}
\textcolor{black}{The kinetic energy is not explicitly written as a functional of the
density. Rather it is a functional of the Kohn-Sham orbitals, the
square sum of which yields $\rho$.} The exchange correlation energy
regroups all of the quantum
effects that result from collapsing the representation 
of the many body Schr\"{o}dinger equation into a functional of the
electronic density $\rho$. 
The ground state energy $E_0 (x)$ is obtained as the minimum value
of $\Ehk(\rho)$ over all admissible electronic densities $\rho$. Its minimum
value is reached at the ground
state electronic density $\rho_0[x]$ corresponding
to the ground state electronic wavefunction $\Psi_0$ of $E_0 (x)$:
\begin{equation*}
\rho_0 [x] =
\argmin_{{\rm admissible} ~{\rho}} \Ehk({\rho}) \quad \text{and} \quad E_0(x) = \Ehk(\rho_0 [x]). 
\end{equation*}
\Cref{fig: electronic density} shows examples of ground
state electronic densities $\rho_0[x]$ computed for several 
planar organic molecules. 
The observed "smearing" of the electronic density between nuclei is responsible for the presence of chemical bonds within the molecule.

\begin{figure}
\center
\subfigure[$\mathrm{C}_2\mathrm{H}_4$]{
\includegraphics[width=1.4in]{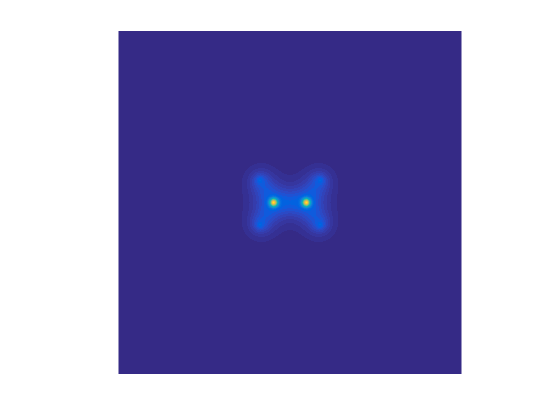}
}
\subfigure[$\mathrm{C}_6\mathrm{H}_4\mathrm{O}\mathrm{S}$]{
\includegraphics[width=1.4in]{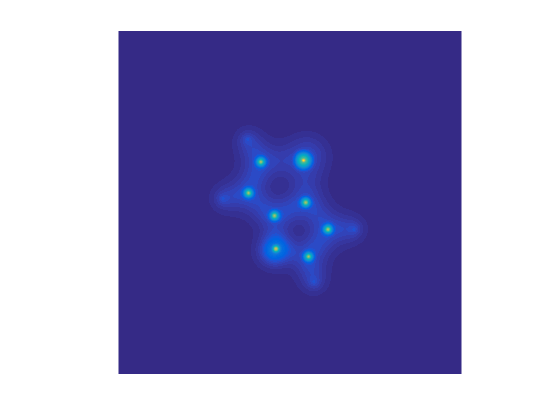}
}
\subfigure[$\mathrm{C}_8\mathrm{H}_9\mathrm{N}$]{
\includegraphics[width=1.4in]{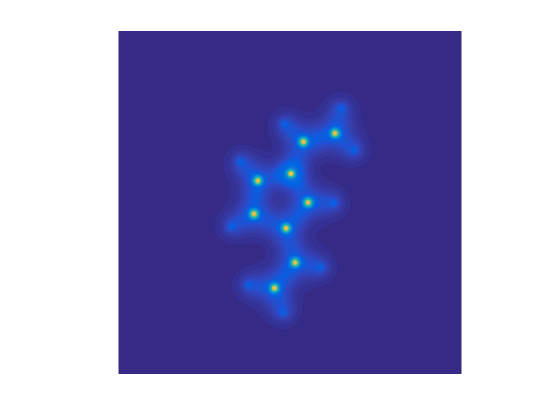}
}
\subfigure[$\mathrm{C}_7\mathrm{H}_9\mathrm{O}\mathrm{N}$]{
\includegraphics[width=1.4in]{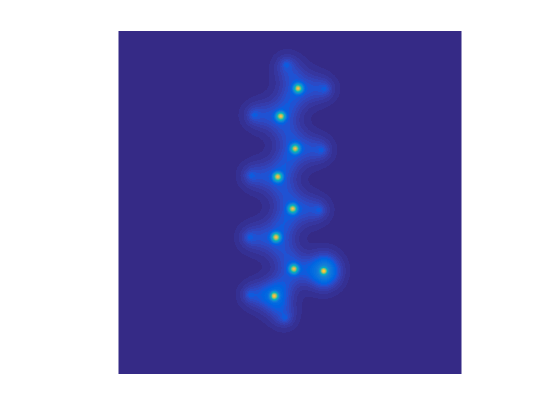}
}
\subfigure[$\mathrm{C}_6\mathrm{H}_3\mathrm{N}\mathrm{O}_2$]{
\includegraphics[width=1.4in]{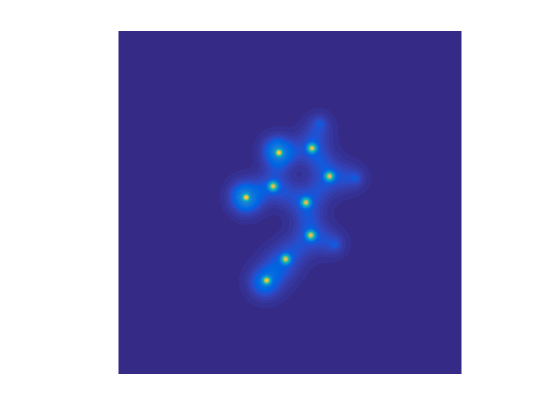}
}
\subfigure[$\mathrm{C}_5\mathrm{H}_6\mathrm{N}_2\mathrm{O}\mathrm{S}$]{
\includegraphics[width=1.4in]{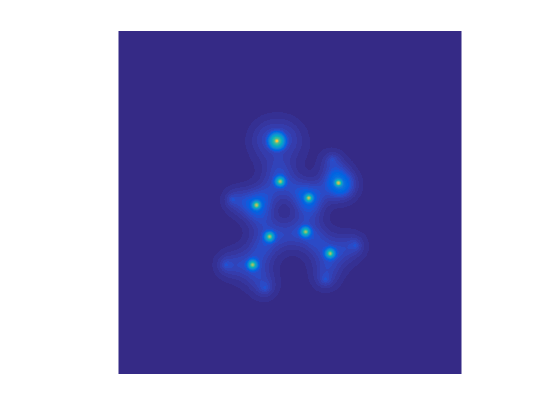}
}
\caption{Ground state electronic density $\rho_0 [x]$ of
planar molecules, in their plane.}
\label{fig: electronic density}
\end{figure}

The Kohn-Sham approach shows that
the high-dimensional Sch\"{o}dinger eigenvalue problem can be replaced
by a variational problem over a three dimensional electronic
density. In practical applications, 
the exchange correlation energy term is approximated, 
which leads to computational errors. However, density functional
theory is most often used to compute molecular energies because it
currently provides the best trade off between accuracy and
computational complexity.

\section{Energy Regression}
\label{Energyregress}

Machine learning regressions are much faster procedures,
which estimate the ground state energy $f(x)$ of a molecular state $x$,
by interpolating a database of ground state energies $\{x_i , f(x_i) \}_i$. 
This approach was pioneered by several groups \cite{Behler:NNPotEnergy2007, rupp:coulombMatrix2012,
  bartok:repChemEnviron2013, vonlilienfeld:fourierQC2015} who obtained
impressive accuracy on different types of molecular databases.
The next section describes Coulomb kernel \textcolor{black}{ridge}
regressions, which give state of the art machine learning results on
organic molecules. In this paper, we shall use sparse regressions,
which are explained in \cref{sec: sparse regression by ols}.

\subsection{Coulomb Kernel Regression}
\label{sec: linear regression with coulomb kernels}

We introduce Coulomb kernel \textcolor{black}{ridge} regressions and
discuss their invariance properties. \textcolor{black}{A kernel ridge regression is a
linear regression over the dictionary $\Phi$,
given by $\tilde f(x)= \lb w, \Phi(x) \rb$, 
where $w = ( w_k )_k$ is computed by minimizing the penalized error:}
\begin{equation}
\label{energy}
\sum_i |\tilde f(x_i) - f(x_i)|^2 + \lambda \, \sum_k |w_k|^2.
\end{equation}
One can verify that the minimum is achieved for a $w$ which can 
be written 
\begin{equation*}
w = \sum_i \alpha_i \, \Phi (x_i),
\end{equation*}
with
\begin{equation*} 
\tilde f(x) = \sum_{i=1}^n \alpha_i\, {K}(x,x_i)~~\mbox{and}~~
{K}(x,x') = \lb \Phi(x),\Phi(x') \rb.
\end{equation*}
The interpolation may thus directly be calculated from the kernel values
${K}(x,x_i)$, by finding the dual variables $\{\alpha_i \}_i$ which minimize
\eqref{energy}. 

For quantum energy regressions, $x = \{r_k,z_k \}_k$ gives the position
and charge of each atom. The kernel must enforce the invariance to isometric
transformations and the Lipschitz continuity to deformations.
A Coulomb kernel 
\cite{rupp:coulombMatrix2012} is calculated from the Coulomb 
energy potentials between all point charges $(r_k,z_k)$
and $(r_l,z_l)$ of $x$:
\begin{equation*}
c_{k,l}(x) = \left\{
\begin{array}{ll}
\frac{1}{2}z_k^{2.4}, & k=l, \\[5pt]
\displaystyle\frac{z_kz_{l}}{|r_k-r_{l}|}, & k \neq l
\end{array}
\right..
\end{equation*}
The resulting Coulomb kernel is
\begin{equation} \label{eqn: coulomb laplace kernel}
{K}(x,x') = \exp \left( - \frac{1}{\sigma} \sum_{k,l}
|c_{k, l}(x) - c_{k,l}(x')| \right).
\end{equation}
If the two molecules have a different number of atoms then the Coulomb
matrix of the smallest molecule is extended with a zero padding. 

Since the Coulomb kernel
depends upon distances $|r_k - r_l|$ it is invariant to isometries applied
to atomic positions. A small deformation modifies these distances by
a multiplicative factor between $1-\epsilon$ and $1 + \epsilon$, where $\epsilon$ measures the deformation size.
Coulomb kernel values are thus modified by multiplicative factors which are
of the same orders, and thus define a representation which is Lipschitz continuous
to deformations.

\textcolor{black}{However, Coulomb kernels are not invariant to permutations on the atom
indices $k$ \cite{PhysRevLett.109.059801}}. This
invariance can be enforced by
ordering the columns and rows of the Coulomb matrix 
in decreasing order of their norm. However, this sorting may not
be uniquely defined for symmetric molecules having several inter atomic
distances of the same value, and it introduces instabilities.
Small perturbations of atomic positions can indeed
modify the ordering, thus breaking the deformation Lipschitz
stability. This instability is reduced by calculating
several Coulomb matrices, by adding a small random noise to the
norms of rows and columns, which modifies their ordering.
The regressed energy of a molecule is the
average of regressed energies calculated over these
different Coulomb matrices.
This numerical technique reduces instabilities but does not eliminate them.
However, these
Coulomb kernel regressions nevertheless
provide good approximations to ground state quantum molecular energies
\cite{rupp:coulombMatrix2012,hansen:quantumChemML2013}.

\subsection{Sparse Regression by Orthogonal Least Squares}
\label{sec: sparse regression by ols}

In this paper, we introduce 
sparse regressions in dictionaries
$\Phi(x) = ( \phi_k (x) )_k$ which are adapted to the properties of 
quantum energy functions $f(x)$. A linear regression is sparse in
$\Phi(x) = ( \phi_k (x) )_k$ 
if there are few, say $M$, non-zero regression
coefficients $w_{k}$. The regression can then be written:
\begin{equation*}
\tilde f(x) = \lb w , \Phi(x) \rb = \sum_{m=1}^M w_{k_m} \, \phi_{k_m} (x).
\end{equation*}
The vector $w$ is optimized in order to minimize the
regression error $\sum_{i \leq n} |f(x_i) - \tilde f(x_i)|^2$, while imposing
that the number of non-zero $w_k$, and hence 
the $\ell^0$ norm of $w$, is smaller than $M$. 
Under appropriate hypotheses on $\Phi(x)$, this $\ell^0$ norm penalization
can be replaced by an $\ell^1$ norm penalization, which is convex. 
However, these hypotheses are violated 
when $\Phi(x) = ( \phi_k (x) )_k$ includes highly correlated functions
\cite{donoho:UPIdealAtomic2001, donoho:sparseGeneralDict2003,
  candes:sparsityIncoherence2007}. This will be the case for quantum 
energy regression over wavelet or Fourier invariants. Sparse regression
may however be computed with non-optimal greedy algorithms such as
greedy orthogonal least square forward selections \cite{chen:OLS1991}.
We describe this algorithm before concentrating 
on the construction of $\Phi(x)$.

A greedy least square algorithm selects the regression vectors
one at a time, and decorrelates the dictionary
relatively to the previously selected vectors. The initial dictionary
$\Phi$ is normalized so that $\sum_i | \phi_k(x_i)|^2 = 1$ for each $k$.
Let us denote by $( \phi^m_k )_k$ the decorrelated dictionary 
at iteration $m$.
We first select $\phi^m_{k_m}$ from among $( \phi^m_k )_k$. We then 
decorrelate the remaining $\phi^m_k$ relatively to $\phi^m_{k_m}$
over the training set $\{x_i \}_i$:
\begin{equation*}
\widetilde{\phi}^{m+1}_{k} = \phi^m_k - \Big( \sum_i \phi^m_{k_m}(x_i)\, \phi^m_k (x_i) \Big)\,
\phi^m_{k_m}.
\end{equation*}
Each decorrelated vector is then normalized to define the updated dictionary:
\begin{equation*}
\phi^{m+1}_k = \widetilde{\phi}^{m+1}_k \, \Big( \sum_i |\widetilde{\phi}_k^{m+1} (x_i)|^2
\Big)^{-1/2}.
\end{equation*}
The linear regression $f_m(x)$ of $f(x)$ is a projection
on the first $m$ selected vectors:
\begin{equation*}
f_m (x) = \sum_{n=1}^m \widetilde{w}_n\, \phi^n_{k_n} (x)~~\mbox{with}~~
\widetilde{w}_n = \sum_i f(x_i)\, \phi^n_{k_n} (x_i).
\end{equation*}
The $m^{th}$ vector $\phi^m_{k_m}$ is selected so that the training 
error $\sum_i |f_m (x_i) - f(x_i)|^2$ is minimized. Since all projections
are orthogonal,
\begin{equation} \label{eqn: ortho error estimate}
\sum_i |f_m (x_i) - f(x_i)|^2 = \sum_i |f(x_i)|^2 - \sum_{n=1}^m |\widetilde{w}_n|^2.
\end{equation}
The error is thus minimized by choosing $\phi^m_{k_m}$ which best
correlates with $f$:
\begin{equation*}
k_m = \argmax_k \left| \sum_i f (x_i) \, \phi^m_{k} (x_i) \right|.
\end{equation*}
The algorithm can be implemented with $QR$ factorization, as described in
\cite{blumensath:OMPvsOLS2007}, or directly as described above. \textcolor{black}{In the
latter case, the cost for each step is $O(nK)$, where $n$ is the
number of training samples $x_i$ and $K$ is the number of dictionary
functions $\phi_k$. Since the algorithm selects $M$ dictionary
functions, the total cost is $O(nKM)$.}

Since $\phi^m_{k_m}$ is a linear combination of the
$\{\phi_{k_n} \}_{n \leq m}$,  
the final $M$-term regression can also be written as a function
of the $\phi_{k_m}$ in the dictionary $\Phi$:
\begin{equation} \label{eqn: final M-term ols regression}
f_M (x) = \sum_{m=1}^M \widetilde{w}_m\, \phi^m_{k_m} (x) = \sum_{m=1}^M w_m\, \phi_{k_m} (x).
\end{equation}
Equation \eqref{eqn: final M-term ols regression} is the orthogonal
least square regression of $f(x)$ over $M$ vectors of  $\Phi(x)$. This
algorithm is used to represent quantum molecular energies over
several dictionaries which are now defined.

\section{Permutation Invariant Electronic Densities}
\label{sec:density}

To compute accurate sparse regressions of $f(x)$ we must introduce a
dictionary $\Phi(x)$ which has the same invariance and regularity
properties as $f(x)$. It must be invariant 
to isometries, and Lipschitz continuous to
deformations, but also invariant 
to permutations of atom indices. 
Similarly to density functional theory, the molecular state $x$ is
represented by an electronic density $\rho[x](u)$ for $u \in \R^3$.
Invariance to isometries is then
obtained by applying an invariant operator $\Theta$:
\begin{equation*}
\Phi (x) = \Theta \rho [x].
\end{equation*}

This section concentrates on the calculation of $\rho[x]$, 
which must be invariant to permutations
of atom indices and covariant to isometries.
We cannot use the ground state molecular density $\rho_0 [x]$,
because its computation is as difficult as
calculating the molecular ground state energy. We replace it
by a much simpler non-interacting density, which adds the electron
densities of isolated atoms. 
It does not include atomic interactions
responsible for molecular bonds.
The effect of these interactions on energy calculations
will be incorporated through regression coefficients on $\Phi(x)$.

Each atom is in a neutral state, with $z_k$ protons and $z_k$ electrons. 
For example, 
Hydrogen and Oxygen atoms correspond respectively to $z_k =1$ and $z_k =8$. 
The electronic density
of such an isolated atom is a rotationally invariant
function $\rho[{z_k}](u)$,
with $\int \rho[{z_k}](u) = z_k$. The resulting non-interacting density of
all atoms located at $\{ r_k\}_k$ is:
\begin{equation} \label{eqn: exact atom rho}
\rho [x] (u) = \sum_{k=1}^K \rho[{z_k}] (u-r_k).
\end{equation}
Clearly $\rho [x]$ is invariant to permutations of the indices $k$,
and covariant to isometries. A deformation
of atomic distances $|r_k - r_l|$ produces a deformation of the atomic
density $\rho$. 

A model for the nuclei is obtained with a point charge 
$\rho[z_k] = z_k \delta(u)$, where each nucleus is treated as a
classical particle. \Cref{sec:classical physics} shows that this model 
captures nuclei Coulomb interactions. As a model for the electrons,
$\rho [z_k] = z_k \delta (u)$ corresponds to aggregating all electrons
at the position of the nucleus, which does not capture their
quantum behavior. Defining $\rho[z_k]$ as the exact electronic
density of an isolated neutral atom with $z_k$ protons is a more
accurate model for the electronic density. It is computed
numerically with a density functional theory software and stored. 

A refinement, which better captures the chemical properties of each
atom is to separate the density of core electrons $\corrho$ from the
density $\valrho$ of valence electrons.
Indeed, chemical bonds are produced by valence electrons, 
whereas core electrons
remain close to the nuclei and do not interact. 
In this case,
$\rho[z_k] (u)$ is considered as a vector with two components
$(\corrho (u) , \valrho(u))$. 
\Cref{sec: numerical results} 
shows that this separation provides a significant
improvement of quantum energy regressions.

\begin{figure}
\center
\subfigure[Ground state electronic density]{
\includegraphics[width=2.25in]{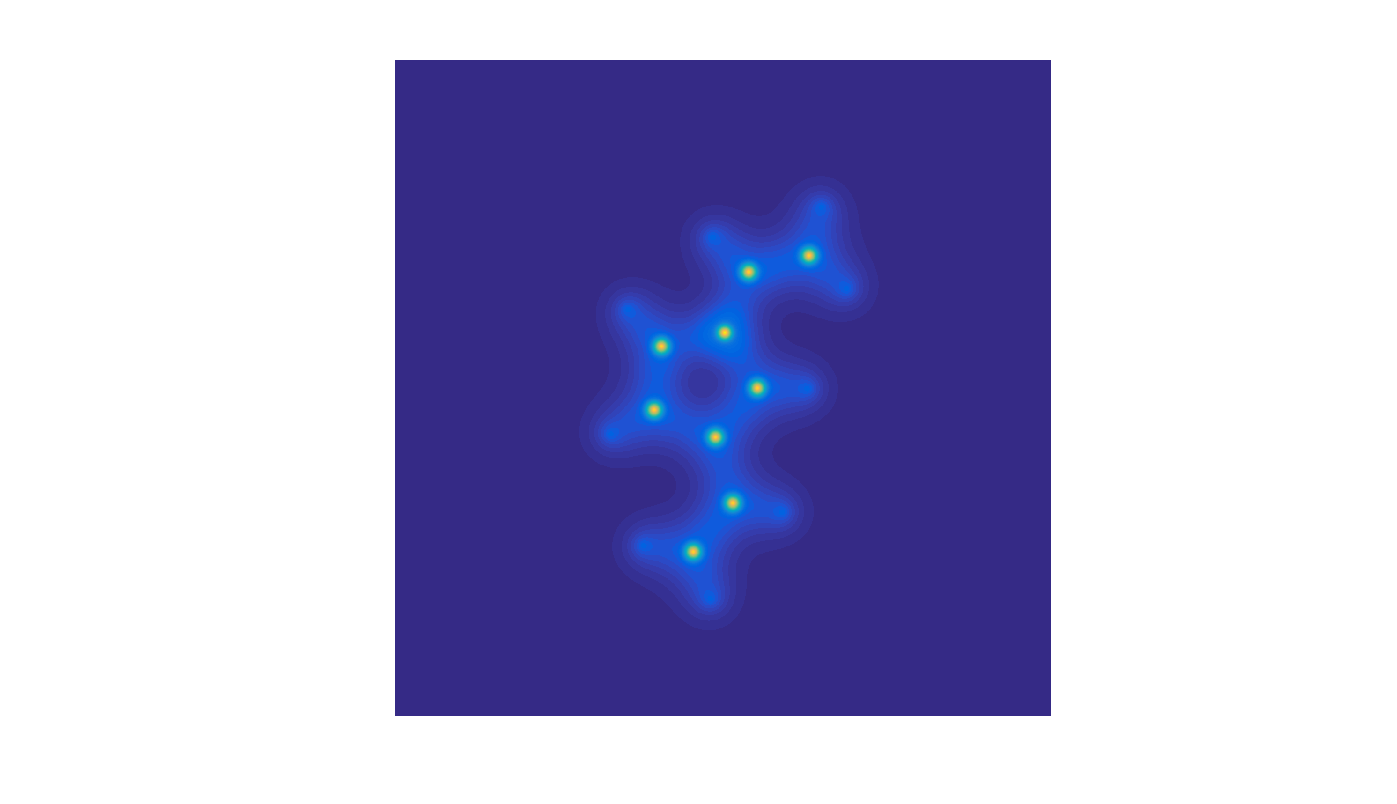}
\label{fig: electronic density model}
}\\
\subfigure[Dirac model]{
\includegraphics[width=2.25in]{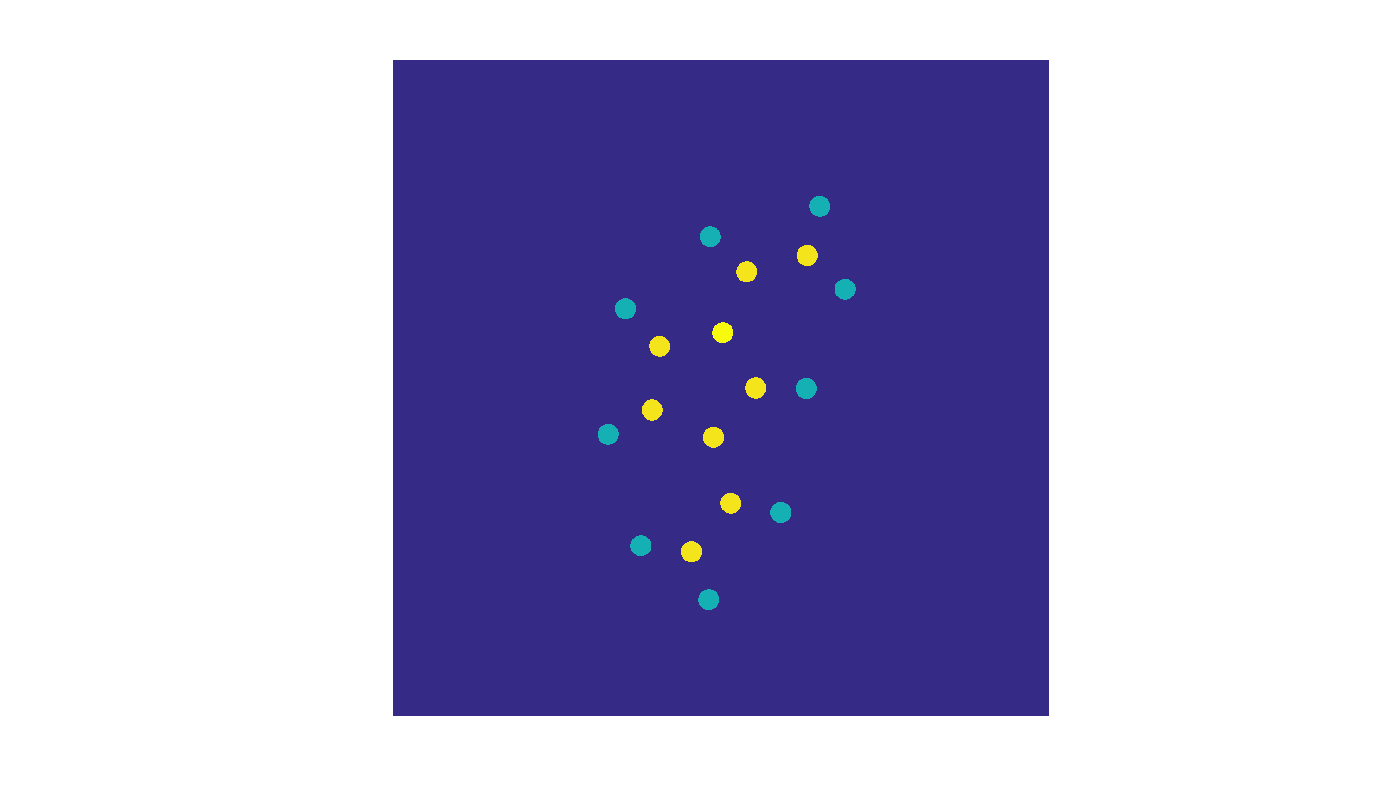}
\label{fig: dirac density model}
}
\subfigure[Atomic density model]{
\includegraphics[width=2.25in]{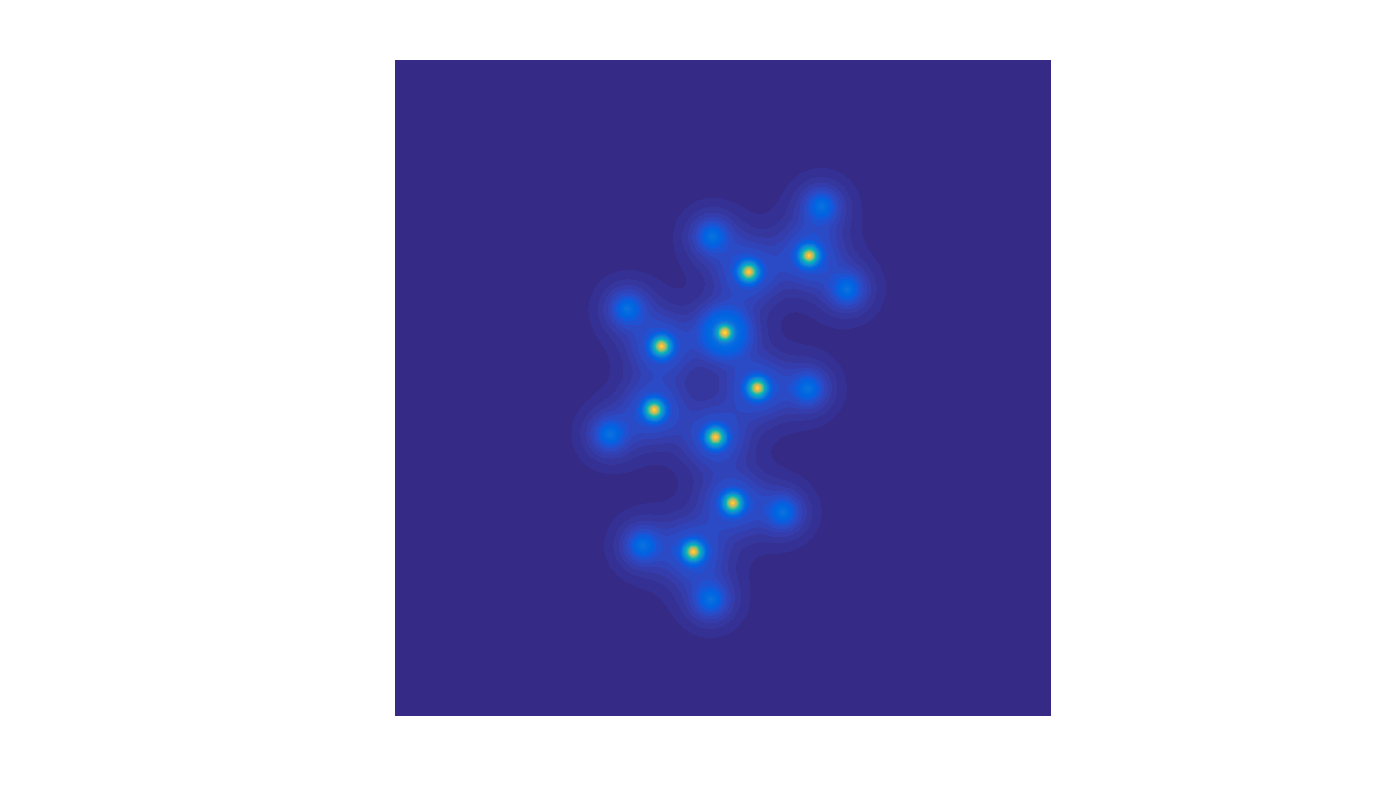}
\label{fig: atomic density model}
}
\subfigure[Core electron density model]{
\includegraphics[width=2.25in]{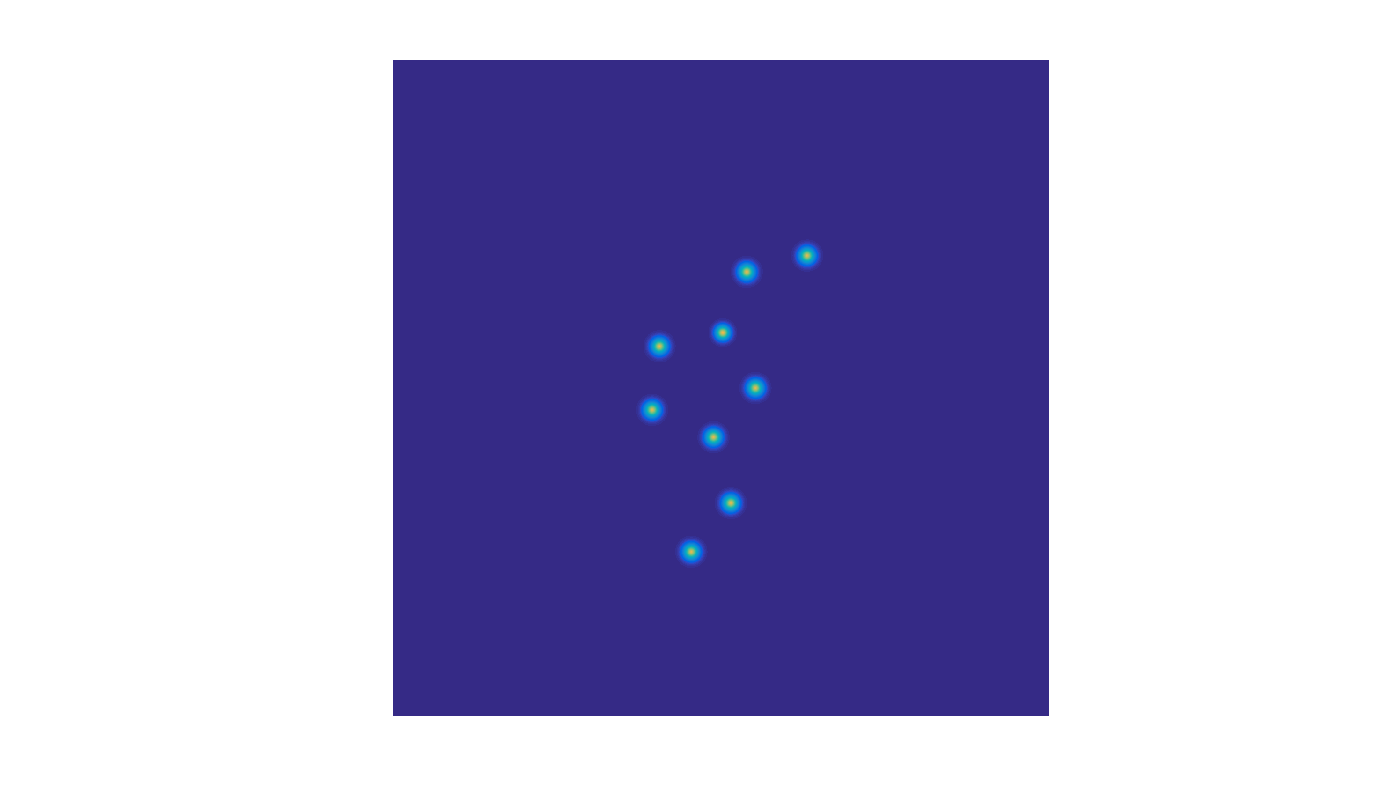}
\label{fig: core density model}
}
\subfigure[Valence electron density model]{
\includegraphics[width=2.25in]{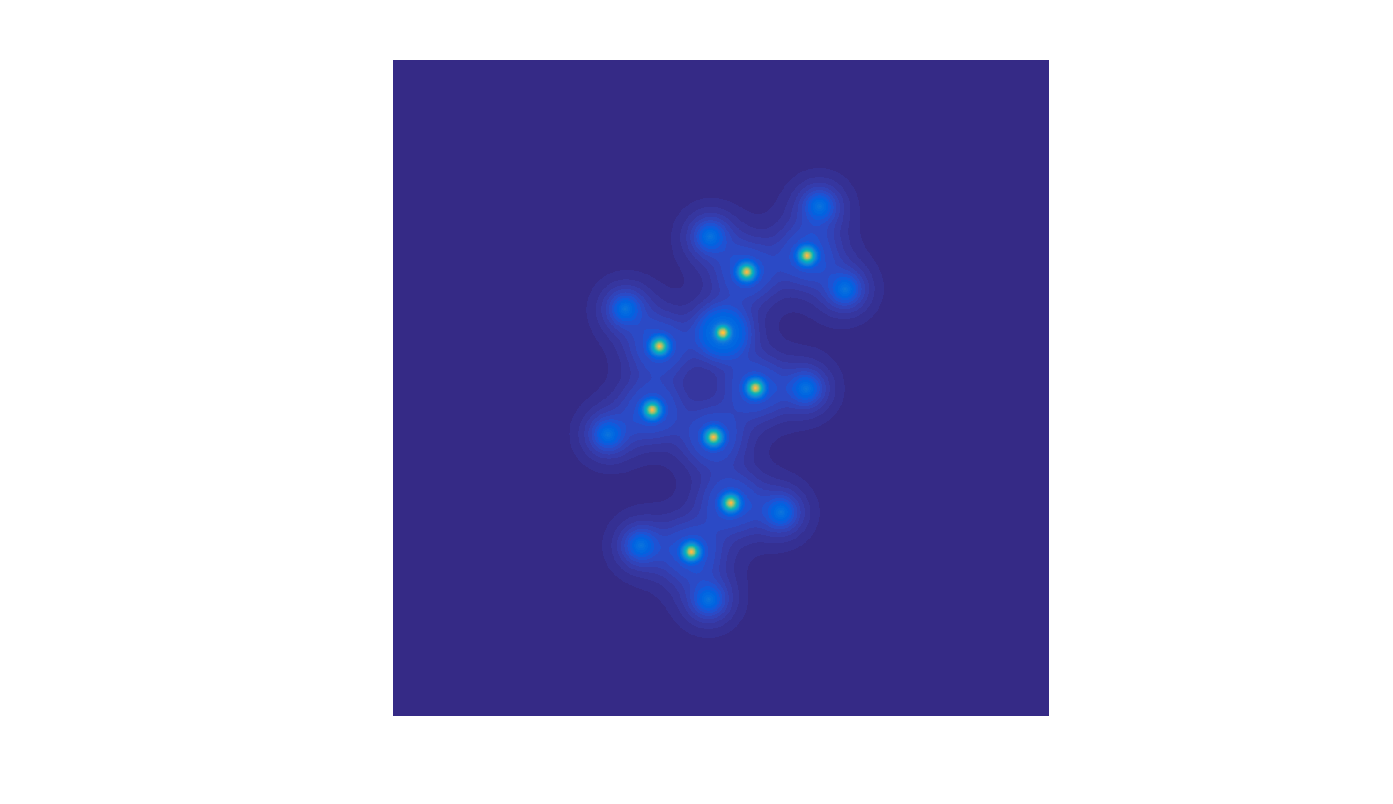}
\label{fig: valence density model}
}
\caption{Electronic density and associated models. \newline (a): Ground state
  electronic density of the planar molecule
  \textcolor{black}{$\text{C}_8\text{H}_9\text{N}$}, restricted to its
  plane, and computed by DFT. \newline (b-e): Non-interacting density models,
  with a Dirac model in (b), an atomic density model computed by DFT
  in (c), which is subdivided into core electrons in (d) and into
  valence electrons in (e).}
\label{fig: exact vs approx density}
\end{figure}

\Cref{fig: electronic density model} shows that the
ground state electronic density $\rho_0 [x]$ is more delocalized
\textcolor{black}{along the bonds}
between atoms than the non-interacting electronic
density in \cref{fig: atomic density model}, computed for individual atoms. 
The core electrons of the density in \cref{fig: atomic density model}
are shown in \cref{fig: core density model}.
They are located close to the nuclei as in
the Dirac model in \cref{fig: dirac density model}. Valence electrons
in \cref{fig: valence density model} are more spread out.

\section{Invariant Fourier Modulus and Autocorrelations}
\label{Fourier}

Invariance to isometries is obtained by applying an operator
$\Theta$ to the electronic density model $\rho$.
A Fourier transform modulus or an autocorrelation
define translation invariant representations.
Integrating over all rotations yields isometric invariant representations,
which are used by several quantum energy regression algorithms
\cite{bartok:gaussAppPot2010,bartok:repChemEnviron2013,bartok:GAP2015,
  vonlilienfeld:fourierQC2015, C6CP00415F}. We review the properties of
these representations. They are invariant to permutations
and isometries, but we show that they are unstable to deformations.
This partly explains the limited
performance of a Fourier modulus representation 
for quantum energy regressions,
as shown by \cref{sec: numerical results}.

The Fourier transform of a density $\rho(u)$ is written for all $\omega \in \R^3$,
\begin{equation*}
\hrho (\omega) = \int_{\R^3} \rho (u) e^{-i u \cdot \omega} \, du.
\end{equation*}
Since the Fourier transform of $\rho(u-\tau)$ is $\hrho (\om)\,e^{-i
  \tau \cdot \om}$,
it results that $|\hrho (\om)|$ is translation invariant. It is also 
symmetry invariant. Indeed, if we define 
$\bar{\rho}(u) = \rho(-u)$ then
$|\hat {\bar{\rho}} (\om)| = |\hat \rho (\om)|$. 

A rotation of $\rho$ yields a rotation of its Fourier transform $\hat \rho$.
A rotation invariant representation is obtained by 
averaging $|\hat \rho (\om)|$
over each rotation orbit, indexed by the sphere $S^2$.
Let us represent $\om \in \R^3$ in spherical coordinates $(\alpha,\eta)$
with $\alpha = |\omega|$ and $\eta \in S^2$. We write $\hat \rho(\om) = \hat \rho_\alpha (\eta)$. 
Rotation invariance is obtained by integrating $|\hat \rho_\alpha (\eta)|^2$ 
over $\eta \in S^2$:
\begin{equation*}
\|\hat \rho_\alpha \|_{2}^2  = \int_{S^2} 
|\hat \rho (\alpha \eta)|^2\, d\eta.
\end{equation*}

For a non-interacting molecular density 
$\rho[x](u) = \sum_k \, \rho[z_k] (u-r_k)$, 
Fourier invariants $\|\hat \rho_\alpha \|_{2}^2$ 
depend upon atomic positions $r_k$
through inter-atomic distances $|r_k - r_l|$. 
Indeed,
\begin{equation*}
|\hat \rho(\om)|^2 = \sum_{k,l} \hat \rho[z_k] (\om)\, \hat \rho[z_{l}] ^* (\om)\,
e^{-i(r_k - r_l) \om}.
\end{equation*}
Since $\rho[z_k] (u)$ is rotationally invariant, 
$\hat \rho[z_k] (\omega)$ is also rotationally invariant. 
Integrating $|\hat \rho (\alpha \eta)|^2$ over $\eta \in S^2$ thus 
only depends on $r_k$ through all $|r_k - r_l|$. 
These Fourier invariants thus define a distance embedding 
which is invariant to index permutations, as opposed to Coulomb
matrices, but they are not stable to deformations.

\Cref{sec:classical physics} proves, essentially, that 
the Coulomb energy of nuclei and electronic densities
can be regressed linearly, with an error $O(\epsilon)$,
over a dictionary of Fourier invariants of size $O(\epsilon^{-2})$.
The radial frequency parameter
$\alpha$ is sampled 
at intervals $\epsilon$ over a frequency range
 $\alpha \in [\epsilon,\epsilon^{-1}]$. It
yields $\epsilon^{-2}$ 
Fourier 
invariants $\{ \|\hat \rho_{k \epsilon} \|_{2}^2 \}_{1 \leq k \leq \epsilon^{-2}}$.
However, quantum energy functionals also include
more complex exchange correlation terms.
The resulting energy stored in so-called chemical bonds grows rather linearly than quadratically
with the number of electrons. To approximate these terms, we include
$\Lu$ norms, which are also invariant to  isometries:
\begin{equation*}
\|\hat \rho_\alpha \|_{1}  = \int_{S^2} 
|\hat \rho (\alpha \eta)|\, d\eta.
\end{equation*}
The resulting Fourier modulus dictionary is defined by:
\begin{equation} \label{FourerRep}
\Theta \rho = \Big( \| \hat{\rho}_{k \epsilon} \|_{1} ~,~ 
\| \hat{\rho}_{k \epsilon} \|_{2}^2 \Big)_{0 \leq k \leq \epsilon^{-2}}~.
\end{equation}
If $\rho$ is defined by core and valence electron densities, then these
norms are computed separately over core and valence densities, which
multiplies by two the number of invariants.

A major drawback of Fourier invariants is their instability to 
deformations. Let $D_\tau \rho (u) = \rho(u-\tau(u))$ be a deformation of
a density $\rho(u)$. 
The amplitude of such a deformation is given by the matrix 
norm of the Jacobian $\nabla \tau$, which is written $\|\nabla \tau \|_\infty$. 
If $\|\nabla \tau \|_\infty < 1$ then $u - \tau(u)$ is a diffeomorphism.  
Lipschitz continuity to deformation means that there exists $C > 0$ such that
for all $\rho \in \Ld(\R^3)$,
\begin{equation}
\label{LIpshc-cont}
\|\Theta \rho - \Theta D_\tau \rho \|_2 \leq C\, \|\nabla \tau \|_\infty\, 
\|\rho\|_2,
\end{equation}
with $\|\rho\|_2^2 = \int |\rho(u)|^2 du$. 
The distance thus decreases to zero when the deformation amplitude
$\|\nabla \tau \|_\infty$ goes to zero. The factor $\|\rho \|_2$ is 
an amplitude  normalization.

It is well known that the Fourier transform modulus 
does not satisfies this Lipschitz continuity property.
Consider a dilation $\tau (u) = \epsilon u$, for which
$\|\nabla u \|_\infty = \epsilon$. In this case 
$\widehat{D_\tau \rho} (\omega) =
(1+\epsilon)^{-1} \hat \rho(\om(1+\epsilon)^{-1})$ so a frequency $\omega_0$ is 
``moved'' by about $\epsilon \, \omega_0$, which is large if $\omega_0 \gg \epsilon^{-1}$. For any $C > 0$, one can then easily construct $\hat \rho$
concentrated around a frequency $\om_0$ so that
$\| |\widehat \rho| - |\widehat {D_\tau \rho}| \|_2$
does not satisfy \eqref{LIpshc-cont}, because the right hand-side 
decreases linearly with $\epsilon$ \cite{mallat:scattering2012}. The rotation invariance
integration does not eliminate this instability so $\Theta$ is not Lipschitz
continuous to deformations.

One may replace the Fourier modulus dictionary by an autocorrelation 
as in SOAP energy regressions
\cite{bartok:gaussAppPot2010,bartok:repChemEnviron2013,bartok:GAP2015,
C6CP00415F}:
\begin{equation*}
A \rho (\tau) = \rho \ast \bar \rho( \tau)~~\mbox{with} ~~\bar \rho(u) = \rho(-u),
\end{equation*}
but it has essentially the same properties.
It is also translation invariant and its Fourier transform
is $|\hat \rho (\om)|^2$.
For an electronic density, we get
\begin{equation*}
A \rho[x](\tau) = \sum_{k,l}  \rho[z_k] \ast \bar \rho[z_l] (\tau - (r_k - r_l)).
\end{equation*}
Rotation invariance is obtained by representing
$\tau$ in spherical coordinates $(\alpha,\eta)$, with $\alpha = |\tau|$,
$\eta \in S^2$, and integrating over $\eta$:
\begin{equation*}
\overline A \rho[x](\alpha) = \int_{S^2} A \rho [x] (\alpha \eta)\, d\eta = \sum_{k,l} 
\int_{S^2} \rho[z_k] \ast \bar \rho[z_l] (\alpha \eta - (r_k - r_l)) \, d\eta.
\end{equation*}
Since each $\rho[z_k]$ is  a rotationally symmetric bump, it
is a sum of bumps centered at $|r_k - r_l|$. It is thus invariant
to isometries acting on $x$. The width of each
bump is below a constant $\sigma$ which measures the maximum spread of
valence electrons of individual atoms. 

Instabilities to deformations 
$D_\tau \rho (u) = \rho (u - \tau (u))$ can be seen with a 
dilation $\tau (u) = (1 + \epsilon) u$. This dilation moves 
each bump located at $|r_k - r_l|$
by $\epsilon |r_k - r_l|$. For long distances
$|r_k - r_l| \gg \sigma / \epsilon$, the bumps produced by the deformed
density have non-overlapping support with the original bumps. It results
that $\|\overline A \rho - \overline A D_\tau \rho \|_2$ does not decreases
proportionally to $\|\nabla u \|_\infty = \epsilon$,
uniformly over all densities $\rho$.
Similarly to the Fourier modulus, autocorrelation invariants
are not Lipschitz continuous to deformations.

Deformation instabilities can be controlled by imposing that the
maximum distance $|r_k - r_l|$ is bounded by a constant which is not
too large. This strategy is adopted by local autocorrelations used by 
Smooth Overlap of Atomic Positions (SOAP) representations
\cite{bartok:repChemEnviron2013}. 
However, this limits the interaction distances
between atoms, which means that one can only regress energies resulting
from short range interactions, while being stable to deformations.

\section{Invariant Wavelet Scattering} 
\label{sec: scattering} 

Wavelet transforms avoid the deformation instabilities of Fourier transforms,
by separating different scales. We first introduce wavelet transform 
invariants while explaining their limitations to regress complex functionals.
Wavelet scattering transforms address these issues
by introducing multiscale interaction terms.

\subsection{Invariant Wavelet Modulus} \label{sec: wavelet transform}

A wavelet transform separates the variations of 
a function along different scales and orientations.
It is computed by
dilating and rotating a mother wavelet $\psi : \R^3 \rightarrow \C$ having a zero average. 
We denote by $r_\theta \in \OO(3)$ an orthogonal operator in $\R^3$,
indexed by a three-dimensional angular parameter $\theta$.
Applying it to $\psi$ at different scales $2^j$ gives:
\begin{equation*}
\psi_{j,\theta}(u) = 2^{-3j}\psi(2^{-j} r_\theta^{-1} u).
\end{equation*}
The wavelet transform of $\rho \in \Ld (\R^3)$ computes 
$\rho \ast \psi_{j,\theta}$.
The Fourier transform of $\psi_{j,\theta}$ is:
\begin{equation*}
\hpsi_{j,\theta} (\om) = \hpsi(2^j r_\theta^{-1} \omega).
\end{equation*}
If $\hpsi(\om)$ has its support centered in the neighborhood of a 
frequency $\omega_0$ then $\hpsi_{j,\theta}$ has a support located in the
neighborhood of $2^{-j} r_\theta \omega_0$. 

Similarly to the Fourier transform, we compute invariants from 
the modulus of wavelet coefficients $|\rho \ast \psi_{j,\theta} (u)|$.
If $\psi^{\ast}(u) = \psi(-u)$, where $\psi^{\ast}(u)$ denotes the
complex conjugate of $\psi (u)$, then $|\rho
\ast \psi_{j,-\theta}| = |\rho \ast \psi_{j,\theta}|$. 
Suppose that $\psi(u)$ is invariant to
two-dimensional rotations around a symmetry axis defined by
$\eta_0 \in S^2$, and 
thus satisfies $\psi(r_\theta u) = \psi(u)$ if $r_\theta \eta_0 = \eta_0$.
All numerical computations are performed with
Morlet wavelets, which are examples of such symmetrical wavelets:
\begin{equation*}
\psi(u) = e^{-|u|^2/2} (e^{i \eta_0 \cdot u} - C),
\end{equation*}
where $C$ is chosen so that $\int \psi(u)\, du = 0$.
The wavelet $\psi_{j,\theta}$ only depends upon the value of
$\eta = r_\theta \eta_0 \in S^2$, modulo a sign. 
All rotations $r_\theta$ which
modify a wavelet can thus be indexed by 
a two-dimensional Euler angular parametrisation $\theta \in 
[0,\pi]^2$ of the half sphere. 

Wavelet coefficients are computed up to a maximum scale $2^J$.
Frequencies below $2^{-J}|\eta_0|$ are captured by a low-pass filter 
$\phi_J(u) = 2^{-3J} \phi(2^{-J} u)$
where $\phi(u) \geq 0$ is a positive
rotationally symmetric function, with $\int \phi(u) du = 1$.
Its Fourier transform
$\hat \phi(\omega)$ is essentially concentrated 
in the ball $|\omega| \leq |\eta_0|$. Often $\phi(u)$ is chosen to be 
a Gaussian. The resulting wavelet transform operator is defined by:
\begin{equation*}
W \rho = \Big\{ \rho \ast \phi_J~,~\rho \ast \psi_{j,\theta} \Big\}_{j
  <  J, \, \theta \in [0,\pi]^2 }.
\end{equation*}

The Fourier transform of $\psi$ is assumed to satisfy the following 
Littlewood-Paley condition for some $0 \leq c < 1$ and
all $\omega \neq 0$:
\begin{equation*}
1-c \leq |\hat \phi(2^J \omega)|^2 + \sum_{j <  J} \int_{[0,\pi]^2}
|\hpsi (2^j r_{\theta}^{-1} \omega)|^2 \, d\theta \leq 1.
\end{equation*}
Applying the Plancherel formula proves that
\begin{equation*}
(1-c)\, \|\rho\|_2^2 \leq \|\rho \ast \phi_J \|_2^2 + 
\sum_{j < J} \int_{[0,\pi]^2} \int_{\R^3}
|\rho \ast \psi_{j,\theta}(u)|^2  \, du \, d\theta \leq  \|\rho \|_2^2.
\end{equation*}
It results that 
$W$ is a stable and invertible operator on $\Ld(\R^3)$.

For an electronic density $\rho[x](u) = \sum_k \rho[z_k](u-r_k)$,
\begin{equation*}
|\rho[x] \ast \psi_{j,\theta}(u) | = \Big|\sum_k \rho[z_k] \ast
\psi_{j,\theta} (u-r_k)\Big|.
\end{equation*}
A wavelet modulus coefficient thus
gives the amplitude at $u$
of interferences produced by wavelets of scale $2^j$ ``emitted'' by
each charge density $\rho[z_k]$ centered at $r_k$; see \cref{fig:
  wavelet modulus first layer visual}.

\begin{figure}
\center
\centerline{\includegraphics[width=8.0in]{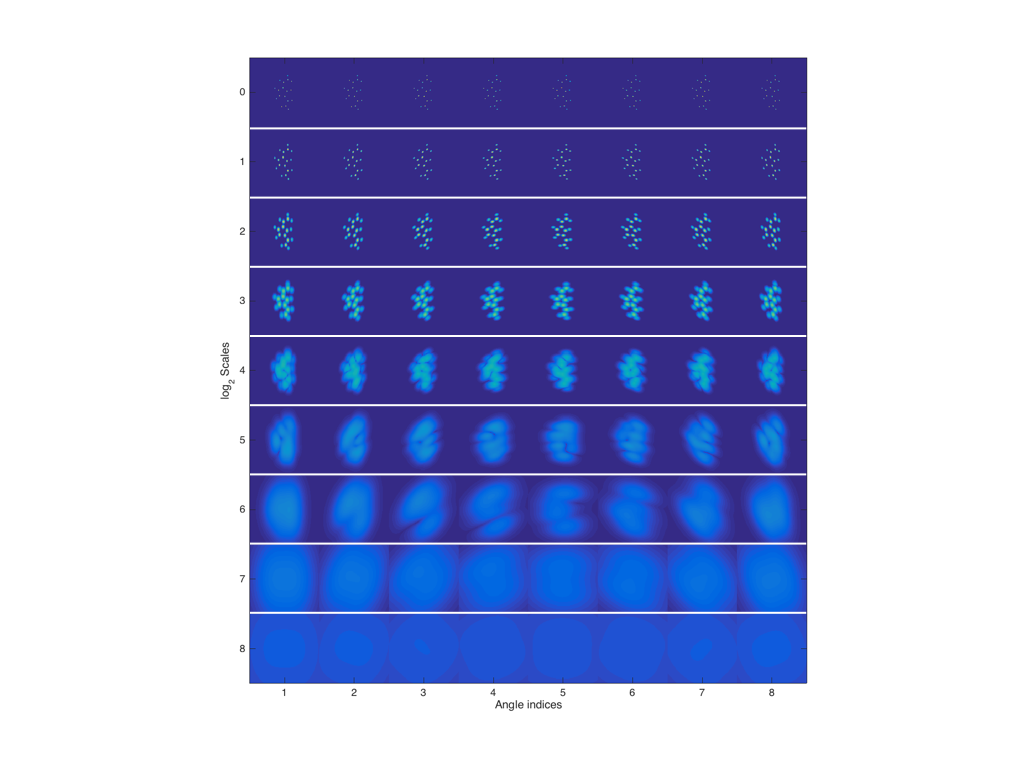}}
\caption{Wavelet modulus transform $| \rho \ast \psi_{j,\theta} |$ of the molecule in \cref{fig: exact vs approx density},
  visualized over $\log_2$ scales $j = 0,
  \ldots, 8$ and eight two dimensional rotations sampled uniformly
  from the half circle $[0, \pi)$. Rotations are indexed on the
  horizontal axis, while scales get progressively larger as one moves down
  the vertical axis. Wavelet invariants such as $\| \rho \ast \psi_{j,
    \cdot} \|_1$ aggregate all of the spatial and rotational
  information from each row into a single wavelet coefficient.}
\label{fig: wavelet modulus first layer visual}
\end{figure}

The wavelet transform modulus is invariant to symmetries. If $\bar \rho(u) = \rho(-u)$ then $|\rho \ast \psi_{j,\theta}(u)| = |\bar \rho \ast \psi_{j,\theta}(u)|$
because $\psi(-u) = \psi^*(u)$. It is 
covariant to translations and rotations but not invariant. 
Translation and rotation invariant dictionaries are
obtained by integrating $| \rho \ast \psi_{j,\theta} (u)|^2$
over the translation and rotation variables:
\begin{equation*}
\|\rho \ast \psi_{j,.} \|_2^2 = \int_{\R^3 \times {[0,\pi]^2}} 
| \rho \ast \psi_{j,\theta} (u)|^2 \, du \,d \theta.
\end{equation*}
The Plancherel formula gives
\begin{equation*}
\|\rho \ast \psi_{j,.} \|_2^2 = 
\frac 1 {(2 \pi)^3} 
\int_{\R^3 \times {[0,\pi]^2}} 
| \hat \rho (\om)|^2\, |\hat \psi (2^j r_{\theta}^{-1} \omega) |^2 \,
d\omega \, d \theta.
\end{equation*}
It computes the energy of $|\hat \rho (\om)|^2$ around the sphere
of radius $2^{-j} |\eta_0|$, over an annulus of width 
of the order of $2^{-j} |\eta_0|$. This frequency
integral is much more delocalized then
the Fourier integral $\|\hat \rho_\alpha \|^2$ which gives
the energy of $|\hat \rho (\om)|^2$ exactly on the sphere of radius $\alpha$.
The index $-j$ can be interpreted as a log radial frequency variable.
These wavelet quadratic invariants have a much 
lower frequency resolution than Fourier invariants
at high frequencies. 

Similarly to the Fourier dictionary \eqref{FourerRep},
the range of radial frequencies is 
defined over an interval $[\epsilon,\epsilon^{-2}]$
for some $\epsilon > 0$. 
Over this interval, there are only 
$3 |\log_2 \epsilon|$  scales defined by 
$ 2 \log_2 \epsilon \leq j < - \log_2 \epsilon = J$. 
For appropriate wavelets, \cref{sec:classical physics}
proves Coulomb energies can be regressed with an $O(\epsilon)$ error 
from these $3 |\log_2 \epsilon|$  wavelet invariants.
Such a regression is much more sparse than a Fourier 
regression which requires $O (\epsilon^{-2})$ terms to obtain an 
$O(\epsilon)$ error. 

To regress molecular energy functionals, which include 
exchange correlation energy terms, like 
the Fourier dictionary \eqref{FourerRep} we include $\Lu$
invariants:
\begin{equation*}
\|\rho \ast \psi_{j,.} \|_1 = \int_{\R^3 \times {[0,\pi]^2}} 
| \rho \ast \psi_{j,\theta} (u)| \, du \,d \theta,
\end{equation*}
at scales $2^j < 2^J$. At the lowest frequencies we get
$\|\rho \ast \phi_{J} \|_1$.
Since $\rho \geq 0$ and $\phi_J \geq 0$ with $\int \phi_J (u) du = 1$,
it results that $\|\rho \ast \phi_{J} \|_1 = \|\rho \|_1$ is the total
electronic charge.
The resulting invariant wavelet dictionary, up to a highest
frequency $\epsilon^{-2}$ is:
\begin{equation} \label{invasnfsdf}
\Theta \rho = \Big\{ \|\rho \|_1~,~\|\rho \ast \psi_{j,.} \|_1 ~,~
\|\rho \ast \psi_{j,.} \|_2^2  \Big\}_{2\log_2 \epsilon \leq j <  J}~.
\end{equation}

As opposed to Fourier, wavelets yield invariants which are stable to 
deformations. It is proved in \cite{mallat:scattering2012} that such wavelet invariants
are Lipschitz continuous to the action of $\bf C^2$ diffeomorphisms. 
Indeed, wavelets are localized in space. Computing the wavelet transform
of a deformed $\rho$ is equivalent to computing the wavelet transform of $\rho$
with deformed wavelets. Since wavelets have a localized support, a deformation
is locally equivalent to a small dilation and rotation, which produces a small
modification of $\Ld$ or $\Lu$ norms. 

\Cref{sec: numerical results} shows that
quantum energy regressions over this wavelet invariant
dictionary produces a 
similar error as regression over Fourier 
invariants \eqref{FourerRep}, but with far fewer terms. The
error is not much reduced because this dictionary does not include enough
invariants to regress the complex behavior of exchange correlation energy terms.

\subsection{Multiscale Scattering Invariants} \label{sec: scattering rep}

A wavelet transform is invertible but integrating its modulus along spatial
and rotation variables does not produce enough invariants.
Wavelet scattering operators have been introduced in \cite{mallat:scattering2012}
to define much richer sets of invariants, 
with higher order multiscale interference terms.

Integrating wavelet modulus coefficients
\begin{equation*}
\|\rho \ast \psi_{j,.} \|_1 = \int_{\R^3 \times {[0,\pi]^2}} 
| \rho \ast \psi_{j,\theta} (u)| \, du \, d \theta,
\end{equation*}
removes the variability of $| \rho \ast \psi_{j,\theta} (u)|$ 
along $u \in \R^3$ 
and $\theta \in {[0,\pi]^2}$. This variability can be captured by computing the wavelet transform
of $| \rho \ast \psi_{j,\theta} (u)|$ 
along both $u$ and $\theta$, which defines a scattering transform
on the roto-translation group
\cite{mallat:rotoScat2013}. In
the following, we concentrate on the variability along $u$ and thus
only compute scattering coefficients along the translation variable $u
\in \R^3$ \cite{mallat:scattering2012}.  

The variations of $| \rho \ast \psi_{j,\theta} (u)|$ 
along $u$ are represented by a second wavelet transform,
which computes convolutions with a second
family of wavelets at different scales $2^{j'}$ with rotations of
angles $\theta + \theta' \in [0,\pi]^2$:
\begin{equation*}
|| \rho \ast \psi_{j,\theta} | \ast \psi_{j',\theta'+\theta} (u)|.
\end{equation*}
Such a coefficient computes the first order interferences
$| \rho \ast \psi_{j,\theta}(u) |$ of the variations 
of $\rho$ at a scale $2^{j}$ along the orientation $\theta$.
These interference amplitudes interfere again by ``emitting'' a new
wavelet of scale $2^{j'}$ whose angle is incremented by $\theta'$.
It yields second order interferences, whose amplitudes 
are measured at each position $u$.
The Fourier transform of 
$| \rho \ast \psi_{j,\theta} (u)|$ has an energy mostly concentrated at frequencies 
$|\omega| < 2^{-j} |\eta_0|$. As a result, the amplitude of these second order coefficients is non-negligible only if $2^{j'} > 2^j$.

Invariance to translations and rotations is obtained by integrating
along $(u,\theta) \in \R^3 \times {[0,\pi]^2}$:
\begin{equation*}
\|| \rho \ast \psi_{j,.} | \ast \psi_{j',\theta'+.} \|_1
= \int_{\R^3 \times {[0,\pi]^2}} 
|| \rho \ast \psi_{j,\theta} (u)| \ast \psi_{j',\theta'+\theta} (u)| \, du \, d \theta.
\end{equation*}
A coefficient $\|| \rho \ast \psi_{j,.}| \ast \psi_{j',\theta'+.} \|_1$
computes the interaction amplitudes of 
variations of $\rho$ at a scale $2^{j}$ along the orientation $\theta$,
located at a distance about $2^{j'}$ within an orientation 
$\theta' + \theta$.
It provides a rich set of invariants, which depend
upon the multiscale geometry of the charge density $\rho$.

Quadratic terms are obtained by integrating the squared amplitude of 
second order coefficients:
\begin{equation*}
\|| \rho \ast \psi_{j,.} | \ast \psi_{j',\theta'+.} \|_2^2
= \int_{\R^3 \times {[0,\pi]^2}} 
|| \rho \ast \psi_{j,\theta} (u)| \ast \psi_{j',\theta'+\theta} (u)|^2 \, du \, d \theta.
\end{equation*}
A scattering transform further iterates this procedure by computing higher
order terms obtained through convolutions with third order wavelets and more.
Because it cascades wavelet transforms, it is proved in \cite{mallat:scattering2012} 
that a scattering transform remains Lipschitz continuous to diffeomorphism 
actions. In this paper, we only keep second order terms for regressions.

Wavelet scales are computed over a frequency range defined by
$2^J > 2^{j'} > 2^j \geq \epsilon^2$.
Adding second order terms to the invariant wavelet 
dictionary \eqref{invasnfsdf} defines a second order wavelet scattering
dictionary:
\begin{eqnarray}
\label{invasnfsdf2}
\Theta \rho &=& \Big\{ \| \rho \|_1\, , \, \|\rho \ast \psi_{j,.} \|_1 ~,~
\|\rho \ast \psi_{j,.} \|_2^2 ~,~ \\
& &
\nonumber
\|| \rho \ast \psi_{j,.} | \ast \psi_{j',\theta'+.} \|_1~,~
\|| \rho \ast \psi_{j,.} | \ast \psi_{j',\theta'+.} \|_2^2
\Big\}_{2\log_2 \epsilon \leq j < j' <  J, \, \theta' \in [0,\pi]^2}.
\end{eqnarray}
This dictionary is invariant to permutations of atom indices, to isometries
and is Lipschitz continuous to deformations. The next section shows that it 
considerably outperforms Fourier modulus and wavelet representations for
quantum energy regressions.

Second layer scattering coefficients are a priori indexed by a four
dimensional parameter $(j, j', \theta') \in \Z \times \Z \times
[0,\pi]^2$. \Cref{sec: scat 2nd layer symmetry} proves that
the two dimensional angle parameter $\theta'$ can be reduced to a one
dimensional angle parameter. Consider the second order term $|| \rho
\ast \psi_{j, \theta} | \ast \psi_{j', \theta' + \theta}|$. If $\eta \in S^2$ is the axis
corresponding to $\theta$ and $\eta' \in S^2$ is the axis
corresponding to $\theta' + \theta$, then after integrating over
$\theta$ the only distinguishing parameter is the angle between $\eta$
and $\eta'$ (which is constant for all $\theta$). Indeed, since
scattering coefficients \eqref{invasnfsdf2} are invariant to rigid
motions, a single rotation applied to both $\eta$ and $\eta'$ will not
change their value. 

\Cref{sec: scat 2nd layer symmetry} also precisely describes the
symmetries of the wavelet $\psi$ and the corresponding wavelet modulus
transform $| \rho \ast \psi_{j, \theta}|$ and second order scattering
transform $|| \rho \ast \psi_{j, \theta}| \ast \psi_{j', \theta' +
  \theta}|$. The approach in the appendix is used to not only prove
the aforementioned result, but also gives insight into how
one can efficiently compute scattering coefficients by fully taking into account the
built in symmetries. 

\section{Numerical Energy Regression of Planar Molecules} 
\label{sec: numerical results}

We compare the performance of Coulomb kernel regressions with Fourier, wavelet
and scattering regressions, on two databases of planar organic
molecules \textcolor{black}{which are in their equilibrium configuration.}
The planar symmetry allows us to reduce computations in two dimensions,
as explained in the next section. 

\textcolor{black}{We remark that in most applications the equilibrium structure of the
molecule is not known, but rather must be found by computing the
partial derivatives of the total energy with respect to the atomic
positions $\{ r_k \}_k$, i.e. the ionic forces. If the dictionary
functions $\{ \phi_k \}_k$ are smooth and the gradient $\nabla \phi_k$
can be computed for each $k$, then the ionic forces can be
incorporated into the regression via:
\begin{equation*}
\nabla \tilde{f}(x) = \sum_k w_k \nabla \phi_k (x).
\end{equation*}
The gradients $\{ \nabla \phi_k \}_k$ can be computed analytically for
the scattering dictionary. We omit the details here, since they are
not used in these numerical experiments. However, such computations
can be used, for example, to extend scattering regressions for the purpose
of calculating equilibrium configurations corresponding to local
minima in the potential energy surface.}

All numerical results can be
reproduced by visiting:
\begin{equation*}
\centerline{\href{https://github.com/matthew-hirn/ScatNet-QM-2D}{\texttt{https://github.com/matthew-hirn/ScatNet-QM-2D}}}
\end{equation*}
and downloading the \texttt{ScatNet-QM-2D} package, which includes all
software and data.

\subsection{Representation of Planar Molecules}
\label{planar}

For planar molecules $x$, the positions $\{ r_k \}_k$ of the 
atom nuclei are in a two-dimensional plane. Since the energy is invariant
to rotations, 
it only depends upon the relative nuclei positions in this plane.
The molecule geometry can thus be represented in two as opposed
to three dimensions. Suppose that the $r_k$ correspond to the two-dimensional
positions of the nuclei in the plane $\R^2$. The electronic density 
$\rho[x] = \sum_k \rho[z_k] (u - r_k)$ can then be restricted to $\R^2$.
The invariant dictionary $\Phi (x) = \Theta \rho[x]$ is 
computed by applying an operator 
$\Theta$ on $\rho[x](u)$, for $u \in \R^2$ as opposed to $\R^3$.
Fourier, wavelet and scattering dictionaries are therefore computed in
two dimensions. Although it does not change the nature of the
regression problem, it reduces computations in numerical experiments. 

There are several ways to restrict the isolated electronic densities
$\rho [z_k]$ to $\R^2$. We derive the two dimensional density from the
three dimensional one by condensing the charge mass on the sphere of
radius $\alpha$ to the circle of radius $\alpha$. Note that since
the isolated charge densities are radially symmetric, the value $\rho
[z_k] (u)$ depends only on $|u|$, i.e., $\rho [z_k] (u) = \rho [z_k]
(\alpha)$ for $\alpha = |u|$. The total charge $z_k$ can therefore be decomposed
over spheres of radius $\alpha$ as follows:
\begin{equation} \label{eqn: charge density sphere}
z_k = \int_{\R^3} \rho [z_k] (u) \, du = \int_0^{\infty} 4 \pi
\alpha^2 \rho [z_k] (\alpha) \, d\alpha.
\end{equation}
A radially symmetric two-dimensional electronic density,
$\rho_{\mathrm{2D}} [z_k] : \R^2 \rightarrow \R$, with total charge
$z_k$, can be decomposed similarly:
\begin{equation} \label{eqn: charge density circle}
z_k = \int_{\R^2} \rho_{\mathrm{2D}} [z_k] (u) \, du = \int_0^{\infty} 2\pi
\alpha \rho_{\mathrm{2D}} [z_k] (\alpha) \, d\alpha.
\end{equation}
Equating the two integrands of \eqref{eqn: charge density sphere} and
\eqref{eqn: charge density circle} yields:
\begin{equation*}
\forall \, \alpha \in [0, \infty), \quad \rho_{\mathrm{2D}} [z_k]
(\alpha) = 2 \alpha \rho [z_k] (\alpha).
\end{equation*}

In two dimensions, reflections are not generated via $-I$ (where $I$ is
the identity matrix) as in $\R^3$, but rather through a fixed reflection about a line,
for example the horizontal or vertical axis. Therefore, if
$\widetilde{\rho}$ is a reflection of $\rho$, 
\begin{equation*}
\| |\widetilde{\rho} \ast \psi_{j, \cdot}|
\ast \psi_{j', \theta' + \cdot} \|_p^p =  \| |\rho \ast \psi_{j,
  \cdot}| \ast \psi_{j', -\theta' + \cdot} \|_p^p,
\end{equation*}
where the angles $\theta, \theta'$ now parameterize the half circle
$S^1$, taken as $(-\pi/2, \pi/2]$. Two dimensional reflections thus reflect
the second order angle parameter $\theta'$, mapping it to
$-\theta'$. Invariant second order scattering coefficients are
obtained by taking the average:
\begin{equation*}
\forall \, \theta' \in [0, \pi/2], \enspace p=1,2, \quad
\frac{1}{2} \sum_{\varepsilon = -1, 1} \| |\rho \ast \psi_{j, \cdot}|
\ast \psi_{j', \, \varepsilon\theta' + \cdot} \|_p^p.
\end{equation*}

In numerical computations, electronic densities
$\rho[x]$ are sampled at sufficiently small intervals $\epsilon$.
In the following, we normalize this sampling interval to $1$.
Electronic densities $\rho[x]$ are sampled over a square of $2^{2J}$ samples,
with $2^J = 2^{9}$.
The Fourier dictionary \eqref{FourerRep} is computed with a two-dimensional 
Fourier transform. The modulus and squared Fourier transform modulus
are integrated over the unit frequency circle $S^1$,
to achieve rotation invariance. The resulting dictionary has $2^{J-1}
= 2^8$ radial Fourier invariants per density channel
(as described in \cref{sec:density}).

The wavelet dictionary \eqref{invasnfsdf}
is computed with a wavelet transform
over $J = 9$ scales $2^j$, $j = 0, \ldots, 8$, with $\Lu$ and $\Ld$ norms.
The angular variable is also integrated over the
circle $S^1$. The wavelet
dictionary thus has $2 J + 1 = 19$ wavelet invariants per density channel.

The scattering dictionary 
\eqref{invasnfsdf2} also includes second order invariants,
which are indexed by two scale indices $0 \leq  j < j' <  J$.
In two dimensions, the angle $\theta$ is discretized over $L$
values in $(-\pi/2, \pi/2]$. In numerical computations, $L = 16$.
Second order terms include $\Lu$ and $\Ld$
invariants. The total number of second
order scattering functions is thus $(L/2 + 1) J (J-1)$.
The wavelet scattering dictionary therefore has
$1 + 2J + (L/2 + 1) J (J-1) = 667$ invariants per density channel.

\textcolor{black}{For the parameters listed previously and using an iMac
  desktop with a 4 GHz Intel Core i7 processor and 32 GB
1600 MHz DDR3 random access memory (RAM), the walltime cost of
computing the $\mathbf{L}^1$ and $\mathbf{L}^2$ scattering
coefficients for a single molecule is approximately 45 seconds. The
wavelet filters $\psi_{j, \theta}$ are implemented as a filter bank,
which requires a one time computation whose walltime is approximately
25 seconds. The wavelet coefficients $\rho \ast
\psi_{j, \theta}(u)$ are computed in frequency, leveraging the fact
that $\widehat{\rho \ast \psi_{j, \theta}}(\omega) =
\widehat{\rho}(\omega) \widehat{\psi}_{j, \theta}(\omega)$. Using the
fast Fourier transform (FFT) and inverse FFT, the cost per molecule is
$O(L J^2 2^{2J})$ to compute the wavelet invariants. Scattering
invariants are computed in a similar fashion, but with a computational
cost of $O(L^2 J^3 2^{2J})$ that reflects the additional second order
invariants. Increasing the angle parameter $L$ increases the accuracy
of the approximate integral over the rotation group and also increases
the number of second order scattering invariants; however, it does
not in general depend on the type of molecule. The scale parameter $J$ grows with
the size of the largest molecule. Typically in two dimensions the relationship is
$J \sim \frac{1}{2} \log_2 K$, where $K$ is the number of atoms, although specific
types of molecules may scale differently (for example, a string of
atoms arranged linearly will scale as $J \sim \log_2 K$, although in
this case one could replace the square grid with a rectangular one).}

\subsection{Numerical Comparison Over Planar Molecules} 
\label{sec: numerical results and comparisons}

We compare the regression performance of Fourier, wavelet and
scattering dictionaries over two databases of planar molecules.
The molecular atomization energies in these
databases were computed using the hybrid density functional
PBE0 \cite{Adamo:PBE0model1999}.
The errors of such hybrid DFT methods are in the range
of $3$ to $5$ kcal/mol \cite{Ramakrishnan:2014aa, kerwin:DFT2016} relative
to experimental energies (mean absolute error). The first database includes
$454$ nearly planar molecules among the $7165$ molecules of the QM7
molecular database \cite{hansen:quantumChemML2013}. We also created a second database
of $4357$ strictly planar molecules, which we denote QM2D. Both databases consist of
a set of organic molecules composed of Hydrogen, Carbon,
Nitrogen, Oxygen, Sulfur and additionally Chlorine in the case of QM2D. The
molecules featured in these databases cover a large spectrum of representative organic groups typically found in chemical compound space.

\textcolor{black}{To evaluate the precision of each regression algorithm, each
data set is broken into five representative folds,
and all tests are performed using five fold cross validation as
described in \cite{hansen:quantumChemML2013}.
We reserve four folds for training and model selection, and the fifth fold for
testing. We do so for each of the five possible permutations, and
report average regression errors over the five test folds with their
standard deviation.} We compute both a mean absolute error
(MAE) which is the average of the absolute value error over the testing fold,
and the root mean-square error (RMSE) which is the square root of the
average squared error.

For the Coulomb matrix kernel regression described 
in \cref{sec: linear regression with coulomb kernels},  
we use a collection of eight Coulomb matrices per molecule, as in
\cite{hansen:quantumChemML2013}. The width $\sigma$ 
of the kernel \eqref{eqn: coulomb laplace kernel} and the ridge
regression parameter $\lambda$ in \eqref{energy} are selected by cross
validation, according to the algorithm described in
\cite{hansen:quantumChemML2013}. The algorithm was validated by recovering the
numerical precision given in \cite{hansen:quantumChemML2013} over the full QM7
database which includes $7165$ molecules. In these experiments, we restrict 
the training database to a subset of the 454 planar molecules in QM7. 
Coulomb regression errors
reported in \cref{fig: MAE} are thus larger than in  
\cite{hansen:quantumChemML2013}.

\begin{table}
\caption{Average Error $\pm$ Standard Deviation over the five folds in
  kcal/mol, with a Coulomb kernel regression and sparse regressions in
  Fourier, wavelet and scattering dictionaries. For scattering
  dictionaries, regressions are computed for several densities
  $\rho[x]$, with a Dirac model, atomic densities calculated by DFT
  and separate core/valence densities.}
\label{fig: MAE}
\footnotesize
\renewcommand{\arraystretch}{1.5}
\centerline{
\begin{tabular}{|c|c||c|c|c||c|c|c|}
\hline
\multicolumn{2}{|c||}{}& \multicolumn{3}{|c||}{2D molecules from QM7}
  & \multicolumn{3}{|c|}{QM2D} \\
\cline{3-8}
\multicolumn{2}{|c||}{}& $\overline M$ & MAE & RMSE & $\overline M$ & MAE & RMSE \\
\hline
\hline
\multicolumn{2}{|c||}{Coulomb Matrix} & N/A & 6.7 $\pm$ 2.8 & 14.8 $\pm$ 12.2 & N/A &
2.4 $\pm$ 0.1 & 5.4 $\pm$ 2.5 \\
\hline
Fourier & Dirac + Core/Valence & 73 $\pm$ 27 & 6.7 $\pm$ 0.7 & 
8.5 $\pm$ 0.9 & 244 $\pm$ 52 & 5.3 $\pm$ 0.2 & 7.2 $\pm$ 0.4 \\
\hline
Wavelet & Dirac + Core/Valence & 38 $\pm$ 13 & 6.9 $\pm$ 0.6 & 9.1
$\pm$ 0.8 & 46 $\pm$ 4 & 5.4 $\pm$ 0.1 & 7.1 $\pm$ 0.3 \\
\hline
Scattering & Dirac & 43 $\pm$ 32 & 15.9 $\pm$ 1.7 & 23.9 $\pm$ 7.2 &
262 $\pm$ 81 & 9.2 $\pm$ 0.1 & 11.7 $\pm$ 0.1\\
& Atomic & 99 $\pm$ 40 & 4.0 $\pm$ 0.5 & 5.2 $\pm$ 0.6 & 299 $\pm$ 56
& 2.7 $\pm$ 0.0 & 3.8 $\pm$ 0.2 \\ 
& Dirac + Atomic & 99 $\pm$ 35 & 4.1 $\pm$ 0.3 & 5.7 $\pm$ 0.6 & 
373 $\pm$ 76 & 2.3 $\pm$ 0.1 & 3.4 $\pm$ 0.3 \\
& Core/Valence & 107 $\pm$ 41 & 3.2 $\pm$ 0.1 & 4.5 $\pm$ 0.2 & 
468 $\pm$ 108 & 1.6 $\pm$ 0.1 & 2.6 $\pm$ 0.5 \\
& Dirac + Core/Valence & 97 $\pm$ 45 & 3.5 $\pm$ 0.2 & 5.0 $\pm$ 0.6 &
450 $\pm$ 91 & 1.6 $\pm$ 0.0 & 2.5 $\pm$ 0.4 \\
\hline
\end{tabular}
}
\end{table}

Sparse $M$-term regressions are computed in Fourier, wavelet
and scattering dictionaries $\Phi (x) = ( \phi_k (x) )_k = \Theta \rho[x]$,
using the orthogonal least-square regression algorithm described in
\cref{sec: sparse regression by ols}:
\begin{equation} \label{eqn: M-term regression model}
\tilde f(x) = \sum_{m=1}^M w_{k_m} \phi_{k_m} (x).
\end{equation}
\Cref{fig:  M-term RMSE curves} compares the value of the
RMSE error as a function of the dimensionality $M$ of the regression model, 
for each dictionary. In each case, the electronic
density $\rho$ is computed with \eqref{eqn: exact atom rho}, for
core/valence densities $\rho[z_k]$ computed by DFT as well as the
Dirac density. Results show
that the RMSE error is similar for 
the Fourier dictionary and the wavelet dictionary, although the
wavelet dictionary needs significantly fewer terms to achieve the same
minimum.

\begin{figure}
\centerline{\includegraphics[width=4.0in]{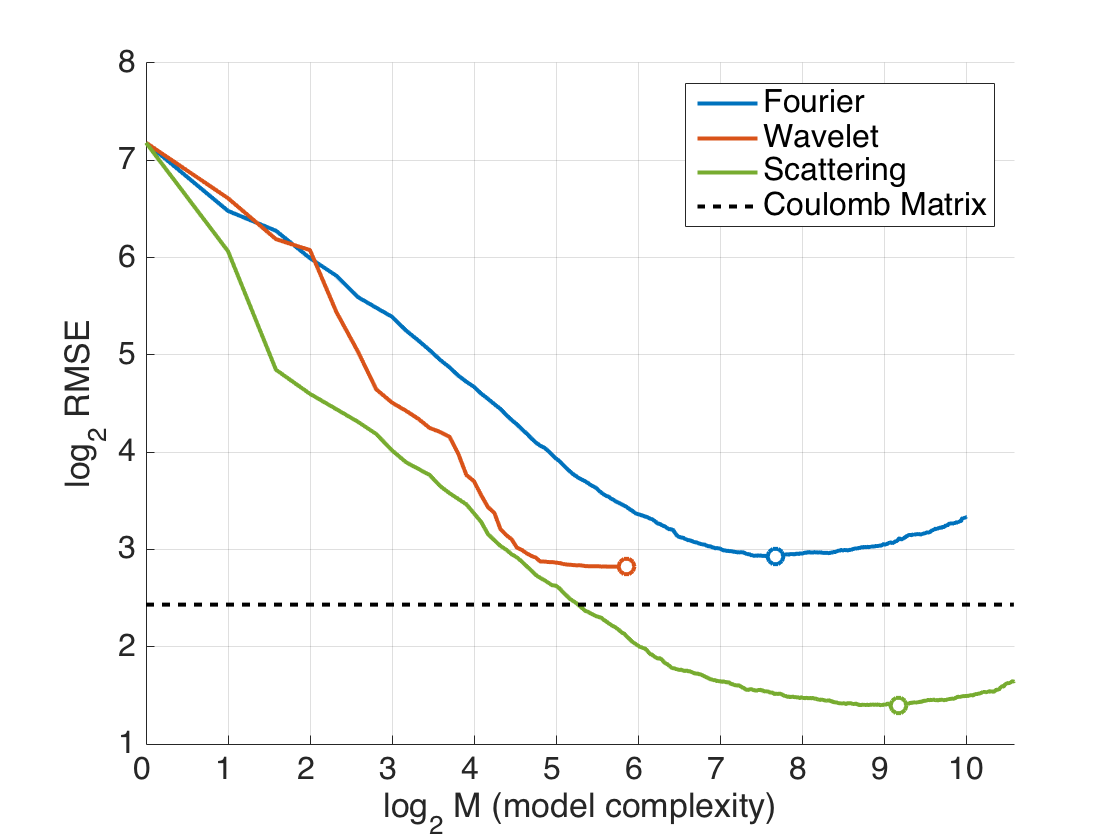}}
\caption{$\log_2 {\rm RMSE}$ of the
root mean square errors on the QM2D database of 4357 molecules, as a
function of $\log_2 M$,
where $M$ is the number of regression terms. Coulomb is 
  the black dashed line, Fourier is blue, 
  wavelet is orange, and scattering is green. The circle
  indicates the optimal value of $M$.}
\label{fig: M-term RMSE curves}
\end{figure} 

The regression error can be decomposed into a bias error due to the inability
to precisely approximate $f(x)$ from a linear expansion in the dictionary
$\Phi(x)$, plus a variance term due to errors when optimizing the weights
$w_k$ from a limited training set. For the Fourier and wavelet
dictionaries, observe that the error
obtained on the planar molecules from the QM7 database of $454$ molecules
is only about 25\% larger than the error on the QM2D database
of $4357$ molecules. On QM2D, it is also significantly bigger than the
error obtained by Coulomb kernel regressions, which saw a drastic
reduction in regression error when going from the smaller to the
larger database. The fact that the overall error for Fourier and
wavelet dictionaries is not significantly
reduced when increasing the database size by a factor of $10$
indicates that the bias error dominates the variance error.
It shows that the dictionary $\Phi(x)$ must be
complemented by more invariants 
to define a better approximation space.

Scattering regressions improve wavelet regressions by adding second
order invariants. When the database size increases
from $454$ to $4357$ molecules the scattering error decreases by
approximately 50\%, and is smaller than the Coulomb kernel regression error. 
The bias term has been reduced to match
the variance term, which gets smaller when increasing the database size.
With separated valence and core electronic densities,
\cref{fig: MAE} shows that the RMSE error of a scattering regression
is $2.6$ kcal/mol. Observe also that the MAE of the 
core/valence scattering regression is below 2.0 kcal/mol on the larger
QM2D database, relative to the DFT energies on which the regression
was trained. The consensus is that DFT has a hard time getting below 2
-- 3 kcal/mol in MAE \cite{kerwin:DFT2016}, and these
errors correspond to very recent methods. Thus if the molecular
energies were experimental and the scattering
regression performed similarly, scattering regressions would be in the
same range of errors as DFT on the QM2D database, which is promising. 

\Cref{fig: MAE} additionally indicates that the standard deviation of the
scattering error is smaller than the standard deviation of the Coulomb matrix
error. Furthermore the largest scattering errors are significantly
smaller than the largest Coulomb matrix errors. The five largest Coulomb matrix
regression errors on individual molecules from QM2D are 224, 185, 83, 70, and
53 (all in kcal/mol), whereas the five largest scattering errors are
62, 47, 40, 27, and 26 kcal/mol. These larger errors may result from
Coulomb matrix instabilities, which originate from the fact that
unsorted Coulomb matrices are not invariant to permutations; see the
end of \cref{sec: linear regression with coulomb kernels} for a
more detailed discussion.

Increasing the model dimension $M$ increases the estimation variance
but reduces the bias. \textcolor{black}{The optimal choice of $M$
results from a bias-variance trade-off, and it is estimated from the
training data. Rather than estimate the
optimal value of $M$ via four fold cross validation on the training data, we apply a
bagging algorithm to learn multiple models, each with their own model
dimension, which are then averaged together to reduce the variance
error.} Each iteration of the bagging algorithm uniformly randomly
selects $\beta \%$ ($0 < \beta < 100$) of the training data to train
the model using the orthogonal least square algorithm, up to a large
value $M_0$ for the model dimension. The algorithm then selects a
model dimension $\overline{M} \leq M_0$ which minimizes the MAE or the
RMSE on the remaining $(100 - \beta)\%$ of the training data. The
resulting $\overline{M}$-term regression is then computed on the test
data. The procedure is repeated $X$ times, and the regressions on the
test data are averaged together to give the final regressed
energies. In numerical experiments $\beta = 90$ for the $454$ nearly
planar molecules from QM7 while $\beta = 80$ for the larger QM2D data
set; the number of bags is $X = 10$ for both data sets. \Cref{fig: MAE}
reports the mean value of $\overline M$ over the $10$ bags, for the
Fourier, wavelet and scattering dictionaries, as well as the
corresponding mean absolute error (MAE) and root-mean-square error
(RMSE), with the standard deviation over the five folds.

For a scattering transform, the model dimension $\overline{M}$ is
approximately $7.5$ times the number of parameters of a molecule.
Indeed, each atom is specified by $3$ variables, 
its two-dimensional position and its charge, so a
molecule of $20$ atoms has $60$ free variables. 
The dominant scales of the wavelets selected by the orthogonal least
square regression go from a small scale, localized around the center
location of each atom, up to a large scale on the order of the
diameter of the molecule. Scales beyond the diameter of the molecule
are used sparingly. It is thus sufficient in the training phase to
compute second order scattering interaction terms at the corresponding
scales $2^j$ and $2^{j'}$, which limits computations. In the testing
phase, only the $\overline{M}$ scattering paths used in the regression
need be computed for new configurations, thus further limiting
computations. 

\begin{figure}
\centerline{\includegraphics[width=3.75in]{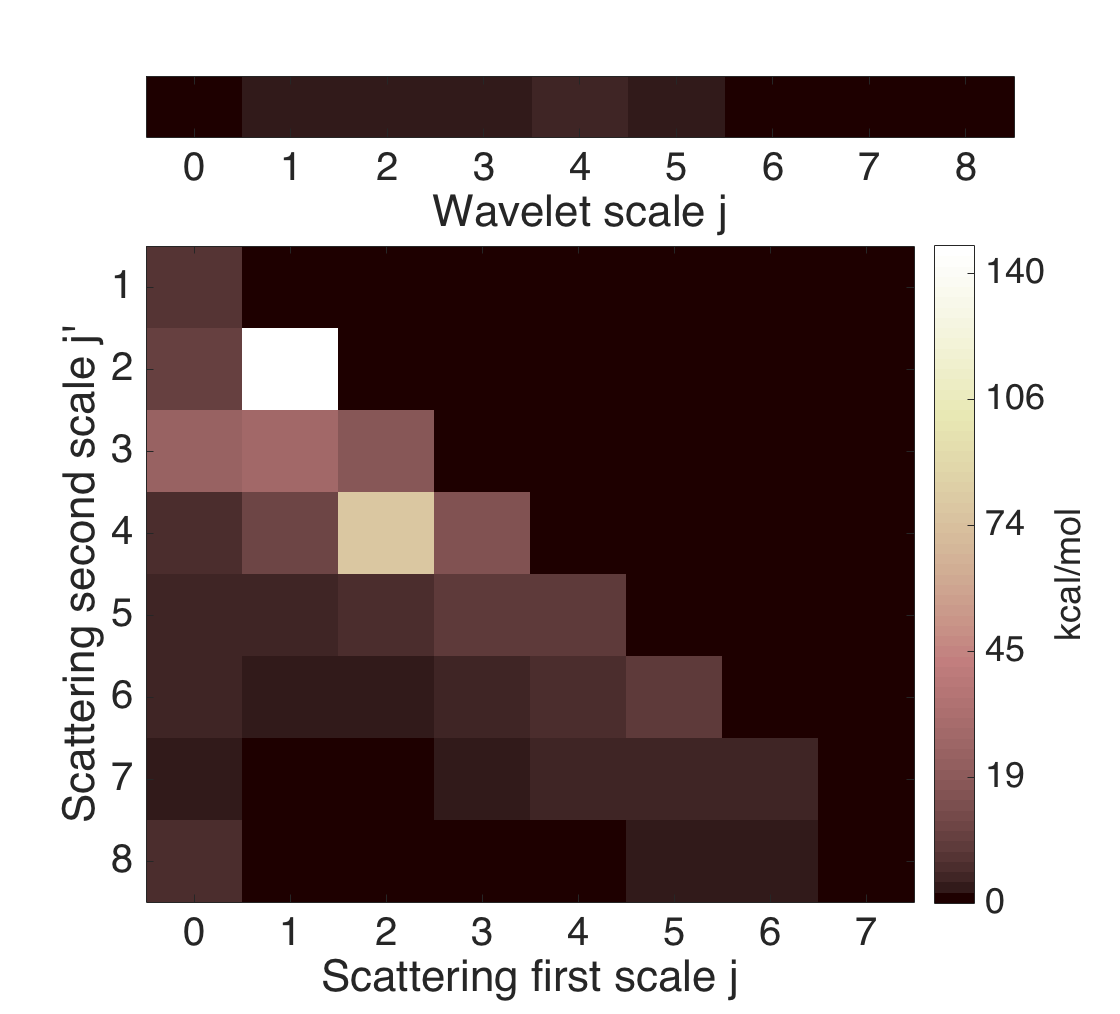}}
\caption{The average scattering weights $\mathbb{E}_l \left(
    \frac{1}{\sqrt{n}} |\widetilde{w}_k^l| \right)$ aggregated by scale. Lighter
  boxes indicate large amplitudes, darker boxes are small amplitudes. The
  upper row corresponds to the weights of wavelet coefficients $\|
  \rho \ast \psi_{j, \cdot} \|_p^p$, while the lower matrix shows the
  weights of scattering coefficients $\frac{1}{2} \sum_{\varepsilon =
    -1, 1} \| | \rho \ast \psi_{j, \cdot} |
  \ast \psi_{j', \, \varepsilon \theta' + \cdot} \|_p^p$ for $j < j'$.}
\label{fig: scat histo}
\end{figure}

We further analyze the distribution of scattering coefficients used in
the regression by computing the expected
value of the squared amplitudes of the weights
$(\widetilde{w}_{k_m})_{m \leq M}$ for the orthonormal scattering dictionary
$(\phi_{k_m}^m)_{m \leq M}$ (\cref{sec: sparse regression by
  ols}), over the larger QM2D database. More
precisely, a scattering regression model \eqref{eqn: final M-term ols
  regression} with $M = 2^9$ nonzero weights was trained from 80\% of
the QM2D database, uniformly randomly selected. The value $M = 2^9$
was selected based on \cref{fig: M-term RMSE curves}, which
indicates it is approximately the optimal value of $M$ for
scattering regressions on the QM2D database. The mean squared
amplitude of each weight $\widetilde{w}_{k_m}$ over the sub-database
$\{ x_i \}_{i \leq n}$ is computed as:
\begin{equation*}
\mathbb{E}_x ( |\widetilde{w}_{k_m} \phi_{k_m}^m (x)|^2 ) =
\frac{1}{n} \sum_{i=1}^{n} |\widetilde{w}_{k_m} \phi_{k_m}^m (x_i)|^2
= \frac{1}{n} |\widetilde{w}_{k_m}|^2.
\end{equation*}
The process was repeated $10^3$ times over different draws of the
training set, yielding a set of weights $(\widetilde{w}_k^l)_{k,l}$,
where $k$ indexes the scattering coefficient and $l = 1, \ldots, 10^3$
indexes the draw. If on the $l^{\text{th}}$ draw a scattering coefficient $\phi_k$ is not selected
within the first $2^9$ iterations of the orthogonal least square
algorithm, the corresponding weight is set to zero, $\widetilde{w}_k^l
= 0$. The resulting $10^3$ root mean squared amplitudes for each $k$ are averaged
over the draws,
\begin{equation} \label{eqn: rmsa over bags}
\mathbb{E}_l \left( \frac{1}{\sqrt{n}} |\widetilde{w}_k^l| \right) =
10^{-3} \sum_{l=1}^{10^3} \frac{1}{\sqrt{n}} |\widetilde{w}_k^l|.
\end{equation}

\Cref{fig: scat histo} displays the values \eqref{eqn: rmsa over
  bags} summed according to their scale, meaning that they were aggregated
over the non-scale scattering parameters corresponding to the density (Dirac,
core, valence), the norm ($\Lu$, $\Ld$), and the angle $\theta'$ for
second order coefficients. From this figure we observe:
\begin{enumerate}[topsep=5pt, itemsep=5pt]

\item
The dominant weights are on second layer scattering coefficients,
giving further numerical validation to their importance. Amongst them,
the weights corresponding to $(j, j') = (1,2)$ have the largest value by
a significant margin. Further numerical investigation shows that the
first scattering coefficient selected in all $10^3$ iterations is $\|
\valrho \|_1$ (not shown in \cref{fig: scat histo}) and it's
corresponding weight is approximately $1475$
kcal/mol in amplitude, which is the order of magnitude of the energy of a single
molecule. The second coefficient selected, however, is always $\|
|\valrho \ast \psi_{1, \cdot}| \ast \psi_{2, 0 + \cdot} \|_1$, with
corresponding mean weight $\mathbb{E}_l \left( \frac{1}{\sqrt{n}}
  |\widetilde{w}_k^l| \right) \approx 129$ kcal/mol, which
constitutes approximately 90\% of the energy of all second order
weights corresponding to scales $(j, j') = (1,2)$.

\item
The majority of the energy in the second layer coefficients
$\frac{1}{2} \sum_{\varepsilon = -1, 1} \| | \rho \ast \psi_{j, \cdot} | \ast
\psi_{j', \, \varepsilon \theta' + \cdot} \|_p^p$ is concentrated along
the band of scales $j < j'$ with $|j - j'| \leq 3$. It can be
interpreted physically as an interaction between two scales, for
example between micro- and meso-scale interactions or meso- and
macro-scale interactions. This type of coupling property was also
observed for multifractal stochastic processes in
\cite{bruna:scatMoments2015}. Practically, it means that to regress
the majority of the energy, only second order scattering coefficients
within the band $| j - j' | \leq 3$ need to be computed. 

\item
Wavelet coefficients $\| \rho \ast \psi_{j, \cdot} \|_p^p$ contain
significantly less energy than second order scattering
coefficients. Amongst these first order coefficients, micro- and
meso-scales ($j = 0, \ldots, 5$) are used more than large scale
wavelets ($j \geq 6$). 

\end{enumerate}

Building upon the first remark, we observe that the first five scattering
coefficients are constant across all $10^3$ draws of the training
set. They are:
\begin{enumerate}[topsep=5pt, itemsep=5pt]

\item
$\| \valrho \|_1$

\item
$\| | \valrho \ast \psi_{1, \cdot} | \ast \psi_{2, 0 + \cdot} \|_1$

\item
$\frac{1}{2} \sum_{\varepsilon = -1, 1} \| | \valrho \ast \psi_{2,
  \cdot} | \ast \psi_{4, \, \varepsilon \frac{5\pi}{16} + \cdot} \|_2^2$

\item
$\| | \valrho \ast \psi_{1, \cdot} | \ast \psi_{3, \frac{\pi}{2} +
  \cdot} \|_1$

\item
$\frac{1}{2} \sum_{\varepsilon = -1, 1} \| | \valrho \ast \psi_{0,
  \cdot} | \ast \psi_{3, \, \varepsilon \frac{3\pi}{16} + \cdot} \|_1$

\end{enumerate}
Observe that after the first coefficient, the next four coefficients
are second order scattering coefficients that all lie within the band
$|j - j'| \leq 3$. Furthermore, amongst the first five coefficients, all
utilize the valence density and four out of five are $\Lu$
norms. These five coefficients serve as a microcosm for trends
observed amongst the first $2^9$ selected scattering coefficients,
some of which were described above in reference to \cref{fig:
  scat histo}, while other trends are described below and in
\cref{fig: aggregate scat magnitudes}. Given that these five coefficients are
universally selected global invariants, it is an interesting open
question, particularly for the second order terms, to understand how
these coefficients relate to the chemical properties of the molecules
in the database. After the fifth selected term,
the orthogonal least square algorithm branches out. There are four possibilities
for the sixth coefficient, although all are $\Lu$ norm
second order coefficients. By the seventh coefficient, both $\Lu$ and
$\Ld$ second order coefficients are possible. An order one wavelet
coefficient is selected, at the earliest, as the $11^{\text{th}}$ coefficient.

Average scattering weights $\mathbb{E}_l \left( \frac{1}{\sqrt{n}}
  |\widetilde{w}_k^l| \right)$ were also summed according to other parameter
dimensions in addition to their scale. In Figure \ref{fig: order 0, 1,
  2 percentages} they are
organized according to the order of the corresponding scattering
coefficient, by aggregating along $k$ the values $\mathbb{E}_l \left(
  \frac{1}{\sqrt{n}} |\widetilde{w}_k^l| \right)$ which
correspond to order zero coefficients of the form $\| \rho \|_1$,
order one coefficients of the form $\| \rho \ast \psi_{j, \cdot}
\|_p^p$, and order two coefficients of the form $\frac{1}{2}
\sum_{\varepsilon = -1, 1} \| | \rho \ast
\psi_{j, \cdot} | \ast \psi_{j', \, \varepsilon \theta' + \cdot} \|_p^p$. The summed
magnitude of each type of coefficient is plotted as a function of the
model dimension $M$. As expected from the previous discussion, the
plot indicates that an order zero
coefficient is selected first, which is in fact $\| \valrho \|_1$, and
that the magnitude of its corresponding weight dominates the
magnitudes of all subsequent weights. Such a coefficient
captures the total valence electronic charge of each molecule. Second
order scattering terms in turn dominate first order wavelet terms,
indicating once again their usefulness. 

\begin{figure}
\centerline{
\subfigure[Order zero, one, and two average scattering weights.]{
\includegraphics[width=3.0in]{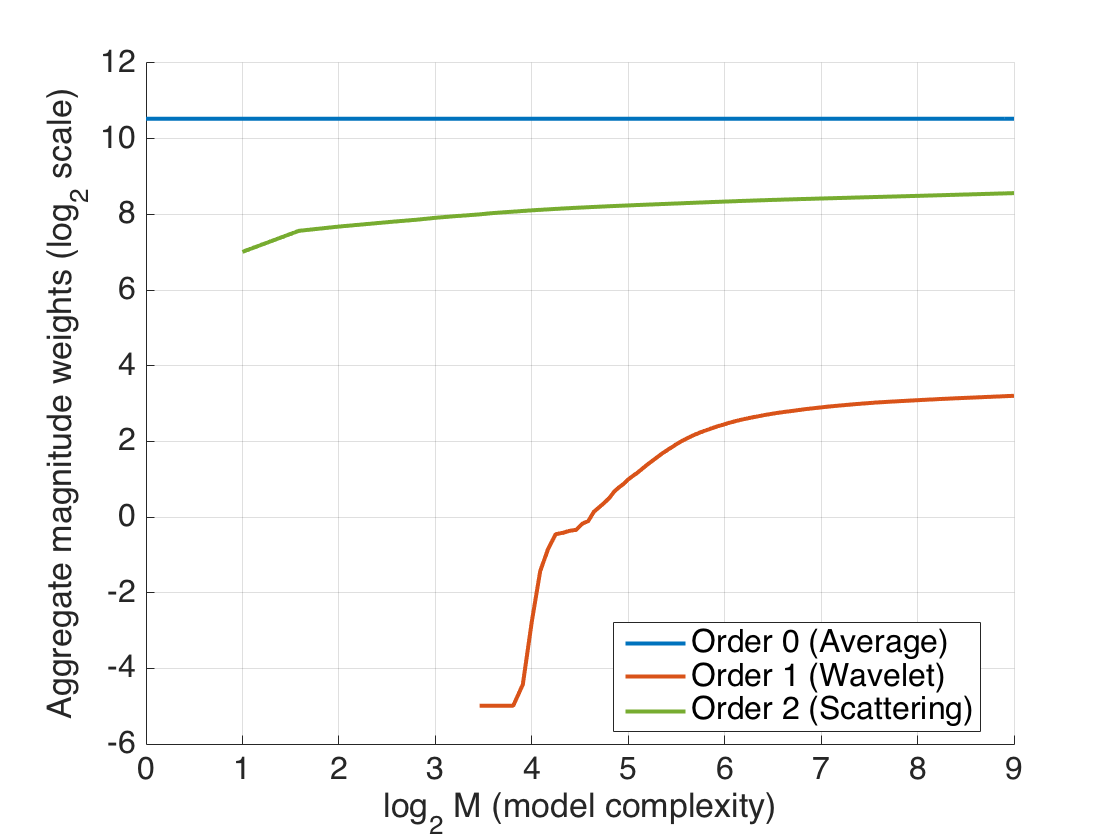}
\label{fig: order 0, 1, 2 percentages}
}}
\centerline{
\subfigure[$\Lu$ and $\Ld$ average scattering weights.]{
\includegraphics[width=3.0in]{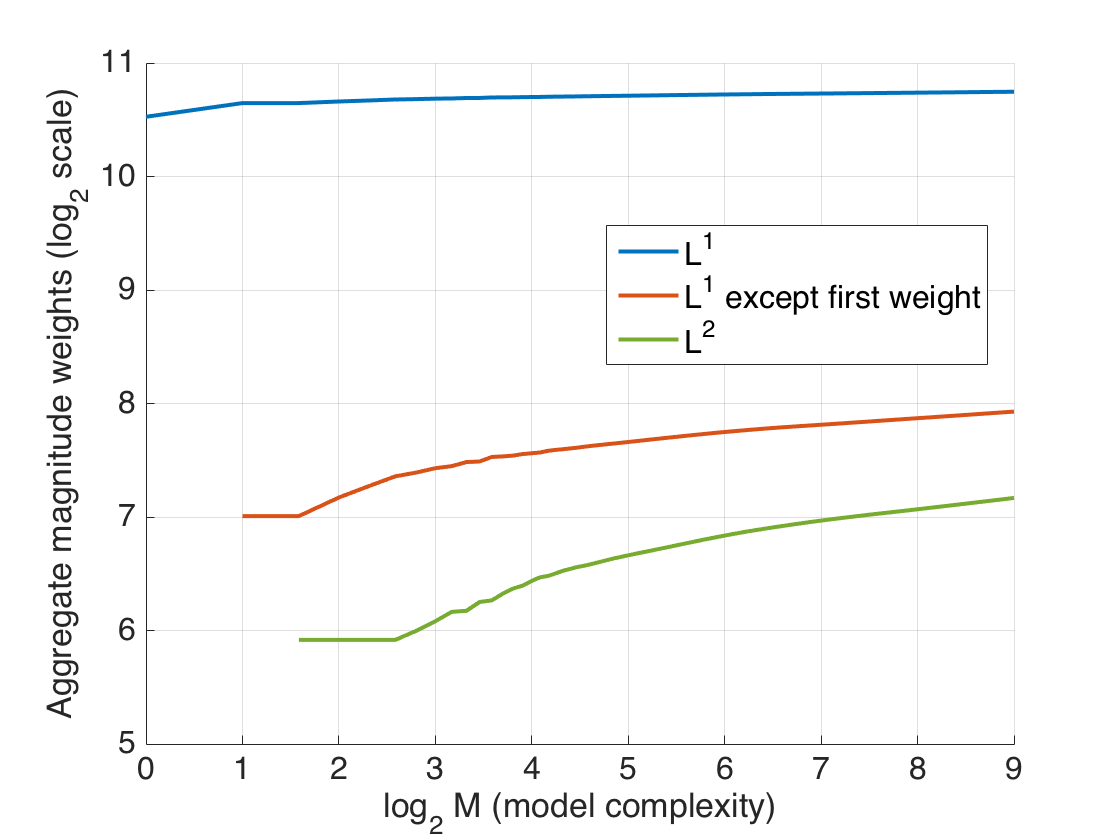}
\label{fig: L1 L2 percentages}
}
\subfigure[Dirac, core, and valence average scattering weights.]{
\includegraphics[width=3.0in]{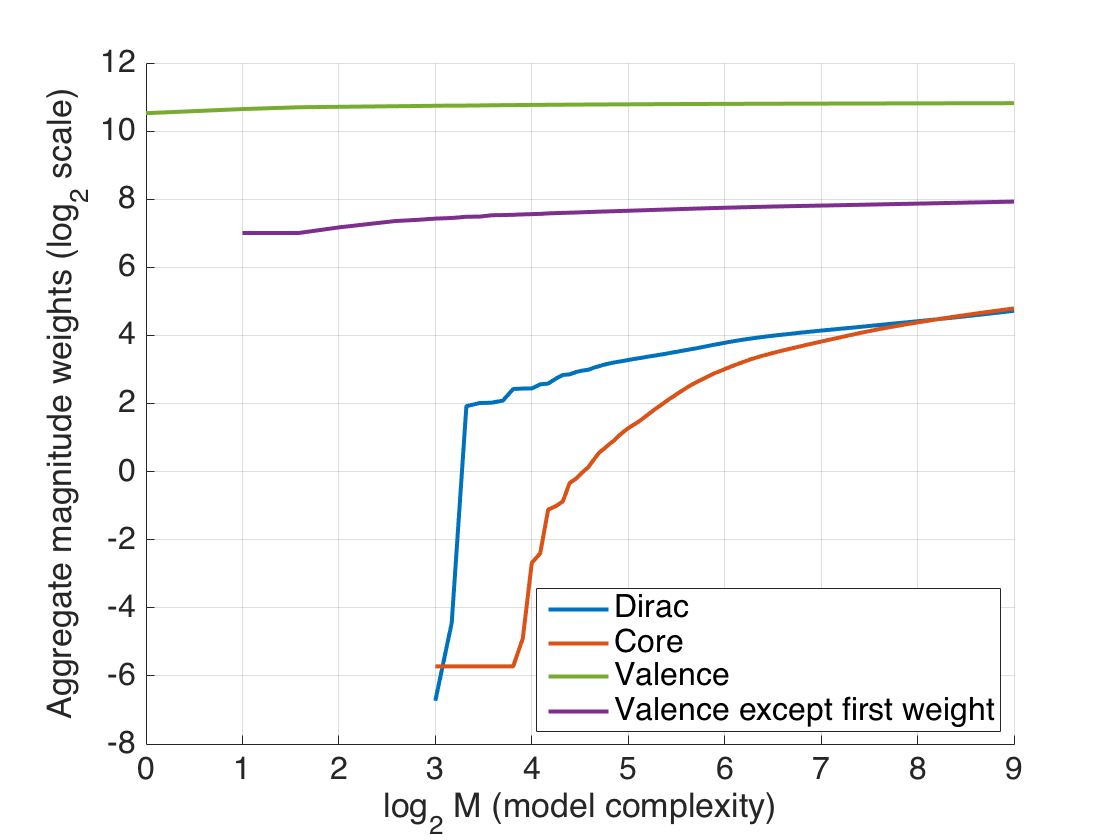}
\label{fig: dirac, core, valence percentages}
}}
\caption{Comparison of the average scattering weights $\mathbb{E}_l \left(
    \frac{1}{\sqrt{n}} |\widetilde{w}_k^l| \right)$ for three different
  categories of the coefficients: the order of the coefficient (zero,
  one, or two), the norm of the coefficient ($\Lu$ or $\Ld$), and the
  type of density (Dirac, core, valence).}
\label{fig: aggregate scat magnitudes}
\end{figure}

Figure \ref{fig: L1 L2 percentages} plots the magnitudes of $\Lu$ and
$\Ld$ aggregated scattering weights. Recall that $\Lu$ coefficients
scale with the number of bonds, while $\Ld$ coefficients incorporate
quadratic behavior resulting from Coulomb interactions. Including the
first weight for the coefficient $\| \valrho \|_1$ shows that $\Lu$
energies dominate $\Ld$ energies. However, when this weight is
removed, Figure \ref{fig:  L1 L2 percentages} indicates that the two
types of coefficients are more balanced, with $\Lu$ weights accounting for
approximately twice the energy as $\Ld$ weights. 

The precision of the regression also depends
upon the electronic density models $\rho[z_k]$ used to 
compute the non-interacting density
$\rho[x] = \sum_k \rho[z_k] (u-r_k)$. 
\Cref{fig: MAE} shows that the error obtained with 
a point charge Dirac density is more than
three times as big as the error obtained with the electronic density of isolated
atoms, computed by DFT. The density model which separates valence and core
electronic densities further reduces the error by more than
30\%. This model multiplies by two the number of invariants 
of each dictionary since invariants
are computed separately for the core and valence densities. 
The primary factor in the improvement is the separated valence
density. Indeed, Figure \ref{fig: dirac, core, valence percentages}
shows that the energy of valence coefficients dominates the energy of
core and Dirac density weights, even if the weight corresponding to
$\| \valrho \|_1$ is removed. This phenomena is
chemically intuitive, since differences between molecules result
primarily from the interactions of valence electrons of constituent
atoms. From \cref{fig: MAE} we also observe that the addition of
the Dirac density, theoretically used to learn
the nuclei Coulomb interactions (\cref{sec:classical physics}),
to either the atomic density or core/valence densities yields only a
small improvement in the numerical precision. It can thus can be
omitted in computations.

\begin{figure}
\centerline{\includegraphics[width=4.0in]{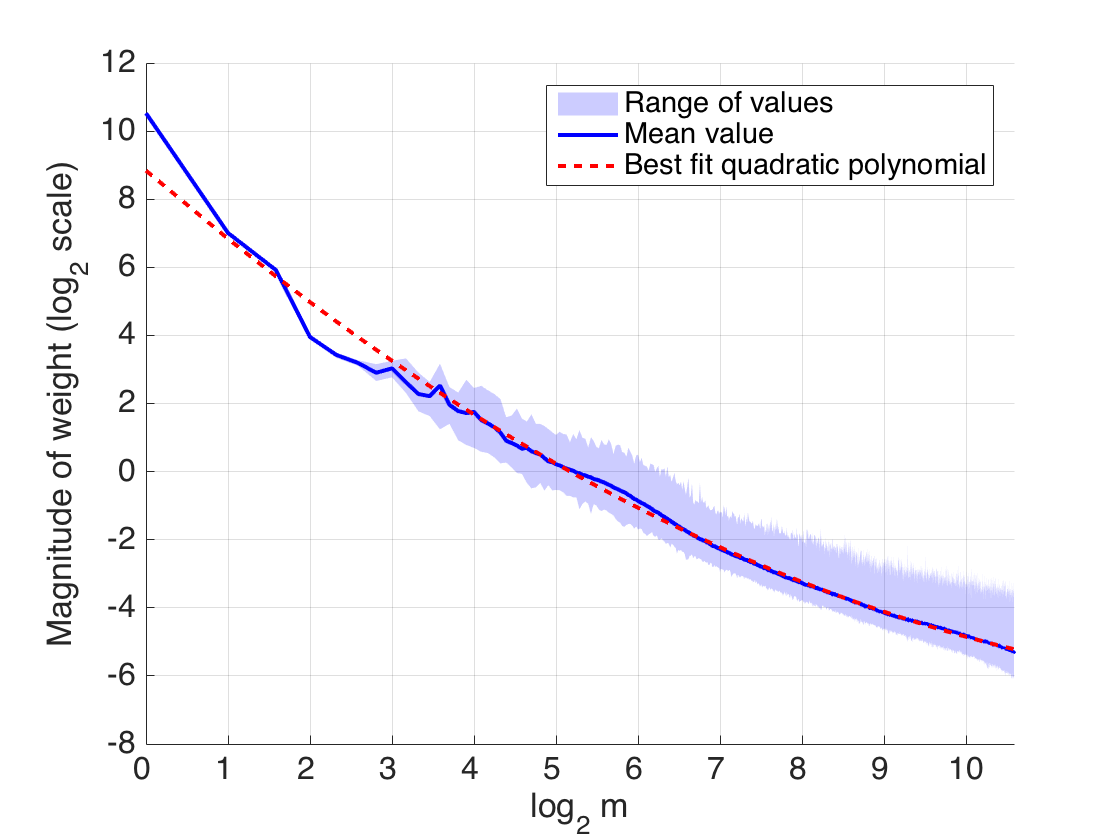}}
\caption{Decay of the scattering weights' magnitude
  $\frac{1}{\sqrt{n}} |\widetilde{w}_{k_m}^l|$ as a function of
  $m$. The solid line is $\mathbb{E}_l \left( \frac{1}{\sqrt{n}}
|\widetilde{w}_{k_m}^l| \right)$, while the upper bound of the shaded
region is $\max_l \left( \frac{1}{\sqrt{n}}
|\widetilde{w}_{k_m}^l| \right)$ and lower bound is $\min_l \left( \frac{1}{\sqrt{n}}
|\widetilde{w}_{k_m}^l| \right)$. The red dashed curve is the
quadratic polynomial that best fits $\mathbb{E}_l \left( \frac{1}{\sqrt{n}}
|\widetilde{w}_{k_m}^l| \right)$ on a $\log_2$ -- $\log_2$ scale.}
\label{fig: scat coeff decay}
\end{figure} 

Finally, we study the decay of the weights' magnitude $|
\widetilde{w}_{k_m}^l |$ as a function of $m$, which by \eqref{eqn:
  ortho error estimate} determines the rate of convergence of
$\tilde{f}$ to $f$ on the training set. We extend the numerical
experiment that generated \cref{fig: scat histo,fig: aggregate scat
  magnitudes} by computing the weights
$\widetilde{w}_{k_m}^l$ up to $m = 3 \cdot 2^9$ (the same maximum
model dimension $M$ in \cref{fig: M-term RMSE curves}). In
\cref{fig: scat coeff decay} we plot $\mathbb{E}_l \left( \frac{1}{\sqrt{n}}
|\widetilde{w}_{k_m}^l| \right)$ as well as $\max_l \left( \frac{1}{\sqrt{n}}
|\widetilde{w}_{k_m}^l| \right)$ and $\min_l \left( \frac{1}{\sqrt{n}}
|\widetilde{w}_{k_m}^l| \right)$ as a function of $m$. From the figure
we observe that on the $\log_2$ -- $\log_2$ scale, the weights decay
similarly to the decreasing part of a convex quadratic polynomial. This implies:
\begin{equation} \label{eqn: variable power law}
\mathbb{E}_l \left( \frac{1}{\sqrt{n}} |\widetilde{w}_{k_m}^l| \right)
\approx 2^c m^{a\log_2 m + b}, \quad a > 0, \enspace b < 0, \quad 1
\leq m \leq 3 \cdot 2^9.
\end{equation}
Equation \eqref{eqn: variable power law}
shows that $\mathbb{E}_l \left( \frac{1}{\sqrt{n}}
  |\widetilde{w}_{k_m}^l| \right)$ satisfies a variable power law
$C|m|^{-\alpha (m)}$, where $\alpha (m)$ decreases linearly as a
function of $\log_2 m$.  For the polynomial plotted in \cref{fig: scat
  coeff decay}, $a \approx 0.07$ and $b \approx -2.1$ and
the polynomial is given by $P(x) = ax^2 + bx + c$. Thus the power
$\alpha (m)$ decreases from approximately $2.1$ at $m = 1$, to
$1.4$ at $m = 2^9$ (the approximate best value of $M$ from \cref{fig:
  MAE,fig: M-term RMSE curves}), to $1.3$ at
$m = 3 \cdot 2^9$. While the initial power law gives a fast rate of convergence, it slows as $m$ gets larger thereby giving
diminishing returns for $\tilde{f}$ as the orthogonal least square algorithm
progresses. \Cref{fig: M-term RMSE curves} indicates that the
diminishing improvement in the bias error of the model is overtaken by the
increasing variance error at approximately $m = 2^9$. 

\section{Regression of Coulomb Energy
  Functionals} \label{sec:classical physics} 

We study the regression of Coulomb potential energies with
Fourier and wavelet invariant dictionaries. Let us first relate
Coulomb energy terms to the quantum energy of a molecule.

The Kohn-Sham approach to density functional theory decomposes the energy
as a sum of four terms in \eqref{eqn: exact energy dft}. 
The ground state energy $f(x)$ is obtained by adding 
the nuclei-nuclei Coulomb potential energy in \eqref{groundstaten},
which can be rewritten: 
\begin{equation*}
f(x) = T (\rho_0) + \frac 1 2 \,U (\nrho - \rho_0) + E_{\mathrm{xc}} (\rho_0),
\end{equation*}
where $\rho_0$ is the ground state electronic density and
$\nrho = \sum_k z_k \, \delta(u-r_k)$ is the point charges
of the nuclei. The first term
$T(\rho_0)$ is the Kinetic energy and $E_{\mathrm{xc}} (\rho_0)$ is the
exchange correlation energy. 
The Coulomb potential energy of
the electronic density and the nuclei charges are regrouped in
\begin{equation}
\label{Coulomb}
U(\rho) = {\rm p.v.} \int_{\R^3} \int_{\R^3} \frac{\rho(u)\, \rho(v)}{|u-v|} \, du \, dv,
\end{equation}
evaluated for the overall charge density $\rho = \nrho - \rho_0$.
Whereas the ground state electronic density $\rho_0$ is Lipschitz with
fast decay, the nuclei density $\nrho$ is a sum of Diracs.
The Coulomb integral is defined as the Cauchy principal value integral
to handle the singularity of the Coulomb kernel $|u|^{-1}$ for point
charges. To simplify notations, we omit the p.v. notation, but all 
Coulomb energy integrals are defined in that sense.

The exchange correlation energy is a complicated term which makes the
mathematical analysis of solutions particularly difficult.
In the following, we concentrate on the Coulomb energy potential
$U(\rho)$. \Cref{CoulombFourier} proves that the Coulomb energy potential
is linearly regressed with an error $O(\epsilon^{1-\delta})$ 
with $O(\epsilon^{-2})$ Fourier invariants calculated from
$\rho$, for arbitrary $\delta > 0$.
\Cref{sec: wavelet coulomb analysis} proves that a wavelet
invariant dictionary produces a regression error $O(\epsilon)$ with 
only $O(|\log \epsilon|)$ terms.

\subsection{Fourier Invariant Regression}\label{CoulombFourier}

The Coulomb potential energy $U(\rho)$ is invariant to isometries. We prove
that it can be linearly regressed over the quadratic 
Fourier invariants introduced in \cref{Fourier}:
\begin{equation*}
\| \hrho_{\alpha} \|_2^2 =
\int_{S^2} | \hrho (\alpha \eta) |^2 \, d\eta.
\end{equation*}

The Coulomb energy \eqref{Coulomb} can be written as a convolution
\begin{equation*}
U(\rho) = {\rm p.v.} \int_{\R^3} \bar \rho \ast \rho (u) \, |u|^{-1}\, du,
\end{equation*}
with $\bar \rho(u) = \rho(-u)$.
Since the Fourier transform of $|u|^{-1}$ is
$4 \pi |\omega|^{-2}$, the Parseval formula gives
\begin{equation}
\label{nsdfoids}
U(\rho) = \frac {4 \pi} {(2 \pi)^3} \int_{\R^3} \frac{|\hat \rho (\om)|^2}{|\om|^2}\, d\om.
\end{equation}
Let $\delta > 0$. The following theorem gives a regression of $U(\rho)$
using $\epsilon^{-2}$ quadratic Fourier invariants 
$\| \hrho_{\alpha} \|_2^2$ sampled on an $\epsilon$--grid of $[\epsilon,
\epsilon^{-1}]$. It proves that the regression 
error is $O(\epsilon^{1-\delta})$. 
Like $\rho_n - \rho_0$, $\rho$ is 
a sum of a Lipschitz function with fast decay and a finite number of point charges.

\begin{theorem} \label{prop: fourier regression}
If $\rho$ is a sum of a Lipschitz function in $\mathbf{L}^1(\R^3)$
with exponential decay and a finite
number of point charges, then for any $0 < \epsilon < 1$ and $0 <
\delta < 1$,
\begin{equation} \label{FourierRegre}
U(\rho) = \frac{\epsilon}{4\pi^2} \left( \| \hat{\rho}_{\epsilon} \|_2^2 +
2 \sum_{k=2}^{\epsilon^{-2}-1} \| \hrho_{k\epsilon} \|_2^2 + \|
\hat{\rho}_{\epsilon^{-1}} \|_2^2 \right) + O(\epsilon^{1-\delta}) \|
\rho \|_1^2, \quad \text{as } \epsilon \rightarrow 0. 
\end{equation}
\end{theorem}

\begin{proof}
Since $\| \hrho_{\alpha} \|_2^2 =
\int_{S^2} | \hrho (\alpha \eta) |^2 \, d\eta$ 
it results from \eqref{nsdfoids} that
\begin{align*}
U(\rho) &= \frac{1}{2 \pi^2} \int_0^{\infty} \frac{1}{\alpha^2}
\int_{S_{\alpha}^2} | \hat{\rho} (\overline{\eta}) |^2 \,
d\overline{\eta} \, d\alpha, \\
&= \frac{1}{2\pi^2} \int_0^{\infty}
\frac{1}{\alpha^2} \cdot \alpha^2 \int_{S^2} |\hat{\rho}(\alpha \eta)|^2 \,
d\eta \, d\alpha = \frac 1 {2 \pi^2}\, \int_{0}^{\infty} \|
  \hrho_{\alpha} \|_2^2 \, d\alpha, 
\end{align*}
where $S_{\alpha}^2 \subset \R^3$ is the sphere of radius
$\alpha$. The following lemma gives a low and high frequency cut  of
this integral to obtain an error $O(\epsilon^{1-\delta})$.

\begin{lemma}\label{lem: coulomb frequency cut}
Under the hypotheses of \cref{prop: fourier regression},
\begin{equation*}
\int_{|\om| < \epsilon} \frac{|\hrho (\omega)|^2}{|\omega|^2} \,
d\omega = O(\epsilon) \| \rho \|_1^2,
\end{equation*}
and for any $0 < \delta < 1$,
\begin{equation*}
\int_{|\om| > \epsilon^{-1}} \frac{|\hrho (\omega)|^2}{|\omega|^2} \,
d\omega = O(\epsilon^{1-\delta}) \| \rho \|_1^2. 
\end{equation*}
\end{lemma}

It results from this lemma that
\begin{equation}
\label{eqn: fourier proof sketch}
U(\rho) = \frac 1 {2 \pi^2}\, 
\int_{\epsilon}^{\epsilon^{-1}} \| \hrho_{\alpha} \|_2^2 \,
  d\alpha + O(\epsilon^{1-\delta}) \| \rho \|_1^2. 
\end{equation}
This second lemma provides a Riemann sum approximation of this
integral using the trapezoid rule.

\begin{lemma}\label{lem: fourier riemann sum}
Under the hypotheses of \cref{prop: fourier regression}, 
\begin{equation}
\label{eqn: fourier proof sketch3}
\int_{\epsilon}^{\epsilon^{-1}} \| \hrho_{\alpha} \|_2^2 \,
  d\alpha =  \frac{\epsilon}{2} \left( \| \hat{\rho}_{\epsilon} \|_2^2 +
2 \sum_{k=2}^{\epsilon^{-2}-1} \| \hrho_{k\epsilon} \|_2^2 + \|
\hat{\rho}_{\epsilon^{-1}} \|_2^2 \right) + O(\epsilon).
\end{equation}
\end{lemma}

Inserting \eqref{eqn: fourier proof sketch3} in \eqref{eqn: fourier
  proof sketch} proves \eqref{FourierRegre}. The proofs of \cref{lem:
  coulomb frequency cut,lem: fourier riemann sum} are given in \cref{sec: proof of coulomb frequency cut,sec: proof of fourier riemann sum}, respectively. 
\end{proof}

Since $\delta$ can be made arbitrarily small, \cref{prop:
  fourier regression} proves, for all practical purposes, that Fourier
regressions can approximate $U(\rho)$ to accuracy $O(\epsilon)$
utilizing $O(\epsilon^{-2})$ invariants. \Cref{lem: fourier
  riemann sum} can be improved by utilizing higher order Newton-Cotes numerical
integration schemes \cite{hoffman:numericalMethods2001}, in which case
one can show that for an arbitrary $\delta' > 0$ independent of $\delta$, there exists
$O(\epsilon^{-(1+\delta')})$ quadratic Fourier invariants that regress
$U(\rho)$ to accuracy $O(\epsilon^{1-\delta})$; some additional
remarks along these lines are given in \cref{sec: proof of
  fourier riemann sum}. Furthermore, the assumptions on $\rho$ are
physically valid. Indeed, it is reasonable to assume $\rho_0$ is Lipschitz, see for example
\cite{fournais:regularityElectronDensity2007}, and it is
known \cite{morrell:electronDensityLongRange1975} that the electronic
density $\rho_0(u)$ decays exponentially fast as $|u| \rightarrow
\infty$. 

\subsection{Wavelet Invariant Regression} \label{sec: wavelet coulomb analysis}

A Coulomb potential energy is computed with the singular 
convolution kernel $|u|^{-1}$.
The Fourier spectrum $4 \pi |\omega|^{-2}$ is singular at $\omega = 0$ and has
a slow decay. Such homogeneous singular operators are better represented
with wavelet transforms
\cite{Daubechies:1992:TLW:130655, meyer:waveletsOperators1993,
  meyer:waveletsCZMO2000, Mallat:2008:WTS:1525499}. The following theorem proves 
that Coulomb potentials have a much more sparse regression
over quadratic wavelet invariants, 
\begin{equation*}
\| \rho \ast
\psi_{j,\cdot} \|_2^2  = \int_{\R^3 \times {[0,\pi]^2}} |\rho \ast \psi_{j,\theta} (u)|^2\,
du \, d\theta,
\end{equation*}
than over Fourier invariants.

\begin{theorem} \label{thm: bandlimited laurent wavelet coulomb regression}
Suppose $\psi$ has compactly supported Fourier transform and satisfies
for $\omega \neq 0$,
\begin{equation}\label{eqn: littlewood-paley exact}
\sum_{j \in \Z} 2^{2j} \int_{{[0,\pi]^2}} | \hpsi (2^j r_{\theta}^{-1} \omega) |^2 \, d\theta = |\omega|^{-2}.
\end{equation}
If $\rho$ is the sum of a Lipschitz function in $\Lu(\R^3)$
and a finite number of point charges, then:
\begin{equation}
\label{satndsfs}
U(\rho) = 4 \pi \sum_{j=2\log_2 \epsilon}^{-\log_2 \epsilon}  2^{2j}\, \| \rho \ast
\psi_{j,\cdot} \|_2^2 + O(\epsilon) \| \rho \|_1^2, \quad \text{as } \epsilon
\rightarrow 0.
\end{equation}
\end{theorem}

\begin{proof}
Since
\begin{equation*}
U(\rho) = \frac{4 \pi}{(2\pi)^3}
 \int_{\R^3} \frac{|\hrho (\omega)|^2}{|\omega|^2} \, d\omega,
\end{equation*}
inserting \eqref{eqn: littlewood-paley exact} gives
\begin{equation*}
U(\rho) = \frac{4 \pi}{(2\pi)^3}
\sum_{j \in \Z} 2^{2j} \int_{{[0,\pi]^2}} 
\int_{\R^3} 
\, {| \hpsi_{j,\theta} (\omega) |^2 }
{|\hrho (\omega)|^2}\,
d\omega \, d\theta .
\end{equation*}
Applying the Plancherel formula proves that
\begin{equation}
\label{Ustatem}
U(\rho) = 4 \pi \sum_{j \in \Z} 2^{2j} 
\| \rho \ast \psi_{j,.} \|_2^2 \, .
\end{equation}
The following lemma proves that the infinite sum over $j$ can
be truncated with an $O(\epsilon)$ error; it's proof is in \cref{sec:
  proof of wavelet cut}.

\begin{lemma}\label{lem: coulomb frequency cut2}
Under the hypotheses of \cref{thm: bandlimited laurent wavelet coulomb
  regression},
\begin{equation}
\label{cutint3}
\sum_{j=-\infty}^{2\log_2 \epsilon}
2^{2j} 
\| \rho \ast \psi_{j,.} \|_2^2 = O(\epsilon) \| \rho \|_1^2,
\end{equation}
and
\begin{equation}
\label{cutint4}
\sum_{j=-\log_2 \epsilon}^{+\infty}
2^{2j} 
\| \rho \ast \psi_{j,.} \|_2^2 = O(\epsilon) \| \rho \|_1^2.
\end{equation}
\end{lemma}

Inserting \eqref{cutint3} and \eqref{cutint4} in \eqref{Ustatem} proves
\eqref{satndsfs}. 
\end{proof}

This theorem proves that Coulomb potential energy can be regressed
to $O(\epsilon)$ error with $3|\log_2 \epsilon| = O (| \log \epsilon|)$ wavelet
invariants, as opposed to $O(\epsilon^{-2})$ Fourier
invariants. Similarly to multipole approximations  of Coulomb potentials
\cite{greengard:multipole1987,greengard:multipoleThesis1988}, the wavelet transform
regroups Coulomb interactions at distances of the order of $2^j$,
which gives an efficient regression formula. 

The wavelet condition \eqref{eqn: littlewood-paley exact} is satisfied if
$\hpsi (\om) = |\omega|^{-1} \, \hpsi^0 (\omega)$ where
$\hpsi^0 (\om)$ satisfies the exact
Littlewood-Paley condition for all $\omega \neq 0$:
\begin{equation*}
\sum_{j \in \Z} \int_{{[0,\pi]^2}} | \hpsi^0 (2^j r_{\theta}^{-1} \omega) |^2 \, d\theta = 1~.
\end{equation*}

In the last two sections, we computed $U(\rho)$ by supposing that
$\rho$ is known. Given the molecular state 
$x = \{z_k,r_k \}_k$,  the pointwise nuclear
charge $\rho_n (u) = \sum_k z_k \delta(u-r_k)$ is known, but the ground
state electronic density $\rho_0 (u)$ is unknown. 
The non-interacting electronic densities 
$\rho[x]$ are crude approximations of $\rho_0$. One may however approximate
$\rho_0$ through deformations of $\rho[x]$. 
Since wavelet transforms
are Lipschitz continuous to deformations, adjusting the linear regression
coefficients in \eqref{satndsfs} can partly take into account the 
electronic density errors. Such adjustments are more unstable with 
Fourier invariant regressions, because of the Fourier instabilities to
deformations at high frequencies.

\section{Conclusion}

We introduced a multiscale dictionary of invariants to compute sparse
regressions of quantum molecular energies. We proved that Coulomb 
potential energies are regressed with few wavelet transform invariants,
but these invariants do not have enough flexibility 
to accurately regress quantum molecular energies. Wavelet 
invariants are complemented by multiscale scattering invariants providing
higher order interaction terms, which improve energy regressions.

Numerical regression errors over two databases of planar molecules are of the
order of DFT errors, but further numerical experiments over
larger databases of three dimensional molecules are needed to evaluate
the range of validity of these quantum energy
regressions. \textcolor{black}{A priori there is no mathematical
  difficulty in working with three dimensional molecules, and indeed
  all of the mathematical analysis is carried out over
  $\R^3$. Numerically, ongoing work includes the development of
  computational and memory efficient three dimensional wavelet
  filters, done in part by leveraging the symmetries described in
  \cref{sec: scat 2nd layer symmetry}.} Already
though these numerical results over planar molecules are opening
mathematical questions to relate more precisely second order
scattering coefficients to the properties of kinetic energy terms and
exchange correlation energy terms of density functional theory.

Let us finally emphasize that scattering energy regressions rely on general
invariance and stability properties, which are common to large classes of
interacting many body problems. 
Multiscale scattering regressions may thus also
apply to other many body problems, which 
exhibit complex multiscale behavior, such as in astronomy,
atmospheric science, and fluid mechanics. 

\appendix

\section{Wavelet Scattering Symmetries} \label{sec: scat 2nd layer
  symmetry}
We prove that second layer scattering terms $\| |\rho \ast \psi_{j,
  \cdot} | \ast \psi_{j', \theta' + \cdot} \|_p^p$, for $p = 1,2$, $j,
j' \in \Z$, and $\theta' \in [0,\pi]^2$, in fact only depend on a one
dimensional angle parameter as opposed to the two dimensional
$\theta'$. In doing so, we give a more detailed explanation of the
wavelet and scattering symmetries described in \cref{sec:
  wavelet transform,sec: scattering rep}, and show how one
might practically implement the scattering transform to take these
symmetries into account. 

To do so, we index rotations of the wavelet $\psi$ by unit vectors
$\eta \in S^2$ as opposed to the Euler angle $\theta$. Recall
$\psi$ is assumed to have two symmetry properties: (i)
$\psi^{\ast}(u) = \psi(-u)$, and (ii) $\psi$ is symmetric about an axis $\eta_0$, meaning $\psi (r
u) = \psi (u)$ for any $r \in \OO (3)$ that fixes $\eta_0$, i.e., $r
\eta_0 = \eta_0$. Define:
\begin{equation*}
\psi_{j,r} (u) = 2^{-3j} \psi (2^{-j} r^{-1} u), \quad j \in \Z, \, r
\in \OO (3),
\end{equation*}
and recall that $\OO (3) = \SO (3) \times \{I, -I\}$. Since we exclusively
consider the wavelet modulus transform, and since property (i) implies
$|\rho \ast \psi_{j,r}| = | \rho \ast \psi_{j,   -r}|$, we can
restrict $\psi_{j, r}$ to $r \in \SO (3)$. Furthermore, let $r_{\eta}
\in \SO (3)$ denote a rotation that moves $\eta_0$ to $\eta$:
\begin{equation} \label{eqn: rotation a onto b}
r_{\eta} \eta_0 = \eta.
\end{equation}
The wavelet $\psi_{j, r_{\eta}}$ is symmetric about $\eta$, and
furthermore
\begin{equation}\label{eqn: wavelet axis symmetry}
\psi_{j, r_{\eta}} = \psi_{j, \tilde{r}_{\eta}},
\end{equation}
for any two rotations $r_{\eta}, \widetilde{r}_{\eta} \in \SO (3)$
defined by \eqref{eqn: rotation a onto b}. Therefore we can index the
rotated wavelet by $\eta \in S^2$, and set:
\begin{equation} \label{eqn: wavelet indexed by S^2}
\psi_{j, \eta} (u) = 2^{-3j} \psi (2^{-j} r_{\eta}^{-1} u), \quad j
\in \Z, \, \eta \in S^2.
\end{equation}

Second layer scattering terms are represented as follows. The second
wavelet modulus transform is computed as:
\begin{equation*}
|| \rho \ast \psi_{j, \eta} | \ast \psi_{j', r_{\eta} \gamma} (u) | =
\left| \left| \int_{\R^3} \rho (v) 2^{-3j} \psi (2^{-j} r_{\eta}^{-1}
    (\cdot-v) ) \, dv \right| \ast \psi_{j', r_{\eta} \gamma} (u) \right|,
\end{equation*}
where $j' \in \Z$ and $\gamma \in S^2$. Similar to the Euler angle
presentation, the orientation of the second wavelet is decomposed as
an increment of the first wavelet orientation. In this case, the axis
$\eta_0$ is first moved to $\gamma$, and then $\gamma$ is moved to a
new axis by the same rotation that moves $\eta_0$ to $\eta$. Isometry
invariant scattering coefficients are derived by integrating over $\eta$
and $u$:
\begin{equation*}
\| | \rho \ast \psi_{j, \cdot} | \ast \psi_{j', r_{\cdot} \gamma}
\|_p^p = \int_{\R^3} \int_{S^2} || \rho \ast \psi_{j, \eta} | \ast
\psi_{j', r_{\eta} \gamma} (u) |^p \, d\eta \, du.
\end{equation*}

In the vector presentation, second layer scattering coefficients are a priori
indexed by $j$, $j'$, and $\gamma \in S^2$. We now
show that they only depend on the one dimensional angle between $\eta$
and $r_{\eta} \gamma$, as opposed to the two dimensional
$\gamma$. First note that this angle is indeed fixed and is equal to
$\arccos (\eta_0 \cdot \gamma)$, since $\eta \cdot
r_{\eta} \gamma = r_{\eta}^{-1} \eta \cdot \gamma = \eta_0 \cdot
\gamma$. Now let $\delta \in S^2$ be a second vector such
that $\eta_0 \cdot \delta = \eta_0 \cdot \gamma$, which
implies that there exists a $r_{\eta_0} \in \SO (3)$ such that
$r_{\eta_0} \delta = \gamma$. Therefore:
\begin{equation*}
r_{\gamma} \eta_0 = r_{r_{\eta_0} \delta}
\eta_0 = r_{\eta_0} \delta =
r_{\eta_0} r_{\delta} \eta_0,
\end{equation*}
for any $r_{\gamma}, r_{\delta} \in \SO (3)$. In particular, we can
take $r_{\gamma} = r_{\eta_0} r_{\delta}$. This yields:
\begin{equation*}
r_{\eta} \gamma = r_{\eta} r_{\gamma} \eta_0 = r_{\eta} r_{\eta_0}
r_{\delta} \eta_0  = r_{\eta} r_{\eta_0} \delta.
\end{equation*}
But $r_{\eta} r_{\eta_0} \eta_0 = \eta$, which implies
$\widetilde{r}_{\eta} = r_{\eta} r_{\eta_0}$. Therefore $r_{\eta}
\gamma = \widetilde{r}_{\eta} \delta$, which completes the proof by
\eqref{eqn: wavelet axis symmetry}. 

The rotation $r_{\eta}$ in \eqref{eqn: wavelet indexed by S^2} can be
any rotation satisfying \eqref{eqn: rotation a onto b}. In practice
one can use the following construction. Let
\begin{equation*}
\nu = \eta_0 \times \eta,
\end{equation*}
where $\times$ denotes the usual cross product. Define the
skew symmetric cross product matrix of $\nu$ as:
\begin{equation*}
[\nu]_{\times} = \left(
\begin{array}{ccc}
0 & -\nu_3 & \nu_2 \\
\nu_3 & 0 & -\nu_1 \\
-\nu_2 & \nu_1 & 0
\end{array}
\right),
\end{equation*}
where $\nu = (\nu_1, \nu_2, \nu_3)$. One can then take $r_{\eta}$ to be:
\begin{equation*}
r_{\eta} = \left\{
\begin{array}{ll}
I, & \eta = \eta_0, \\[5pt]
I + 2 [\nu]_{\times}^2, & \eta = -\eta_0, \\[5pt]
I + [\nu]_{\times} + \dfrac{1 - \eta_0 \cdot \eta}{\| \nu \|^2} \, [\nu]_{\times}^2, & \text{otherwise}.
\end{array}
\right.
\end{equation*}
If ones takes $\eta_0 = (1, 0, 0)$, then by the above proof, second
layer terms are only distinguished by the value $\eta_0 \cdot \gamma =
\gamma_1$, where $\gamma = (\gamma_1, \gamma_2, \gamma_3)$. Therefore
one can parameterize the second layer with $(j, j', t) \in \Z \times
\Z \times [0,1]$, utilizing coefficients:
\begin{align*}
\| | \rho \ast \psi_{j, \cdot} | \ast \psi_{j', r_{\cdot} \gamma (t)}
\|_p^p &= \int_{\R^3} \int_{S^2} || \rho \ast \psi_{j, \eta} | \ast
\psi_{j', r_{\eta} \gamma (t)} (u) |^p \, d\eta \, du, \\
&\text{where } \gamma (t) = (1-t,  \sqrt{t(2-t)}, 0), \enspace t \in [0,1].
\end{align*}

\section{Coulomb Regression Lemmas from \Cref{sec:classical physics}}

\subsection{Proof of \cref{lem: coulomb frequency
    cut}} \label{sec: proof of coulomb frequency cut}

Recall \cref{lem: coulomb frequency cut} from \cref{CoulombFourier}.

\begin{lemma}[Restatement of \cref{lem: coulomb frequency cut}]
Let $\rho$ be the sum of a Lipschitz function in $\Lu (\R^3)$ and a finite number of point charges. Then:
\begin{enumerate}[nolistsep]
\item
There exists a constant $C$ such that for any $\epsilon > 0$,
\begin{equation}
\label{cutint1}
\int_{|\om| < \epsilon} \frac{|\hrho (\omega)|^2}{|\omega|^2} \,
d\omega \leq C \cdot \| \rho \|_1^2 \cdot \epsilon.
\end{equation}

\item
For each $0 < \beta < 1$, there exists a constant $C$ such that for
any $0 < \epsilon < 1$,
\begin{equation}
\label{cutint2}
\int_{|\om| > \epsilon^{-1}} \frac{|\hrho (\omega)|^2}{|\omega|^2} \,
d\omega \leq C \cdot \| \rho \|_1^2 \cdot \epsilon^{\beta}.
\end{equation}

\end{enumerate}
\end{lemma}

\begin{proof}[Proof of \eqref{cutint1}]
Since $\rho \in \Lu (\R^3)$, we have $\hat{\rho} \in
\mathbf{L}^{\infty}(\R^3)$ with $\| \hat{\rho} \|_{\infty} \leq \|
\rho \|_1$. Using this fact with a change to spherical coordinates yields:
\begin{align*}
\int_{|\omega| < \epsilon} \frac{|\hat{\rho} (\omega)|^2}{|\omega|^2}
  \, d\omega &\leq \| \hat{\rho} \|_{\infty}^2 \int_{|\omega| <
               \epsilon} |\omega|^{-2} \, d\omega, \\
&\leq \| \rho \|_1^2 \int_{0}^{2\pi} \int_{0}^{\pi}
  \int_{0}^{\epsilon} \alpha^{-2} \cdot \alpha^2
  \cdot \sin \vartheta \, d\alpha \, d\vartheta \, d\varphi, \\
&\leq C \cdot \| \rho \|_1^2 \cdot \epsilon.
\end{align*}
\end{proof}

The proof of \eqref{cutint2} is more delicate and will be broken
into two cases. Recall that $\rho = \nrho - \rho_0$ where $\rho_0 \in
\Lu (\R^3)$ is Lipchitz, and $\nrho$ is the sum of a
finite number of point charges. Thus:
\begin{equation*}
\int_{\omega > \epsilon^{-1}}
\frac{|\hat{\rho}(\omega)|^2}{|\omega|^2} \, d\omega  =
\int_{|\omega| > \epsilon^{-1}} \frac{
  |\hat{\rho}_{\mathrm{n}}(\omega)|^2 + |\hat{\rho}_0(\omega)|^2 -
  \hat{\rho}_{\mathrm{n}}(\omega) \hat{\rho}_0^{\ast}(\omega) -
  \hat{\rho}_{\mathrm{n}}^{\ast} (\omega) \hat{\rho}_0 (\omega)}{|\omega|^{2}} \, d\omega. 
\end{equation*}
Additionally using H\"{o}lder's inequality,
\begin{equation*}
\left| \int_{|\omega| > \epsilon^{-1}}
  \frac{\hat{\rho}_{\mathrm{n}}(\omega)
    \hat{\rho}_0^{\ast}(\omega)}{|\omega|^2} \, d\omega \right| \leq
\left( \int_{|\omega| > \epsilon^{-1}} \frac{|\hat{\rho}_{\mathrm{n}}
    (\omega)|^2}{|\omega|^2} \, d\omega \right)^{\frac{1}{2}} \left(
  \int_{|\omega| > \epsilon^{-1}} \frac{|\hat{\rho}_0
    (\omega)|^2}{|\omega|^2} \, d\omega \right)^{\frac{1}{2}}.
\end{equation*}
Thus we have separated the proof of \eqref{cutint2} into two
sub-cases, one for Lipschitz functions and one for sums of point
charges. We begin with the former.

\begin{lemma} \label{lem: Lip cut}
If $\rho \in \Lu (\R^3)$ is Lipschitz, there exists a
constant $C \geq 1$ such that for any $0 < \epsilon < 1$,
\begin{equation*}
\int_{\omega > \epsilon^{-1}}
\frac{|\hat{\rho}(\omega)|^2}{|\omega|^2} \, d\omega \leq C \cdot
\| \rho \|_1^2 \cdot \epsilon.
\end{equation*}
\end{lemma}

\begin{proof}[Proof of \cref{lem: Lip cut}]
Since $\rho$ is Lipschitz, there exists $C' \geq 1$ such that
\begin{equation} \label{eqn: rho lip fourier decay}
|\hrho(\omega)| \leq \frac{\sqrt{C'} \| \rho \|_1}{1+|\omega|}.
\end{equation}
Therefore:
\begin{align*}
\int_{|\omega| > \epsilon^{-1}}  \frac{|\hrho
  (\omega)|^2}{|\omega|^2} \, d\omega &\leq C' \| \rho \|_1^2 \int_{|\omega| >
\epsilon^{-1}} (1 + |\omega|)^{-2} \cdot |\omega|^{-2} \, d\omega, \\
&= C' \| \rho \|_1^2 \int_{0}^{2\pi} \int_{0}^{\pi} \int_{\epsilon^{-1}}^{\infty} (1 +
  \alpha)^{-2} \cdot \alpha^{-2} \cdot \alpha^2 \cdot \sin \vartheta
  \, d\alpha \, d\vartheta \, d\varphi,\\
&\leq C \cdot \| \rho \|_1^2 \cdot \epsilon.
\end{align*}
\end{proof}

For point charges we have:

\begin{lemma} \label{lem: point charge cut}
If $\rho$ is the sum of a finite number of point charges and $0 <
\beta < 1$, there exists a constant $C$ such that for any $0 <
\epsilon < 1$,
\begin{equation*}
\int_{\omega > \epsilon^{-1}}
\frac{|\hat{\rho}(\omega)|^2}{|\omega|^2} \, d\omega \leq C
\cdot \| \rho \|_1^2 \cdot \epsilon^{\beta}.
\end{equation*}
\end{lemma}

For the proof of \cref{lem: point charge cut} we define a
function $\phi: \R^3 \rightarrow \R$ with prescribed decay in space and smoothness in
frequency in order to approximate the point charges. The function
$\phi$ is defined through its Fourier transform. First define the
following bump function:
\begin{equation*}
b (\omega) = \left\{
\begin{array}{ll}
\exp \left(- \frac{1}{1 - |2\omega|^2} \right), & \text{for } |\omega| < 1/2, \\
0, & \text{for } |\omega| \geq 1/2.
\end{array}
\right.
\end{equation*}
Then take:
\begin{equation*}
\hphi (\omega) = C \cdot b \ast b (\omega),
\end{equation*}
where $C$ is chosen so that $\hphi (0) = 1$. By construction, $\hphi
(0) = 1$, $0 \leq \hphi \leq 1$, $\supp \, \hphi \subset \{ |\omega| <
1 \}$, and $\hphi \in C^{\infty} (\R^3)$. In space, one has $\phi \geq
0$, $\int \phi = 1$, and for each $n \in \N$, there exists a constant
$C_n > 0$ such that
\begin{equation*}
\phi (u) \leq \frac{C_n}{1 + |u|^n}.
\end{equation*}

Define the $s$-dilation of $\phi$ as:
\begin{equation*}
\phi_s (u) = s^{-3} \phi (s^{-1} u), \quad s > 0,
\end{equation*}
and note that:
\begin{equation*}
\int_{\R^3} \phi_s(u) \, du = 1, \quad 0 \leq \hphi_s \leq 1, \quad \text{and} \quad \supp \, \hphi_s
\subset \{ |\omega| \leq s^{-1} \}.
\end{equation*}

\begin{proof}[Proof of \cref{lem: point charge cut}]
Since $\rho$ is the sum of point charges we write it as:
\begin{equation*}
\rho (u) = \sum_k z_k \delta (u-r_k).
\end{equation*}
 Let $V(u) = |u|^{-1}$ so that $\widehat{V}(\omega) = 4\pi
 |\omega|^{-2}$. Using $\hphi_{\epsilon}$ and the Parseval
 formula we bound the quantity of interest:
\begin{align}
\int_{|\omega| > \epsilon^{-1}} \frac{|\hat{\rho}
  (\omega)|^2}{|\omega|^2} \, d\omega &\leq \int_{\R^3} \frac{|\hat{\rho}
  (\omega)|^2}{|\omega|^2} (1 - \hphi_{\epsilon} (\omega)) \,
d\omega, \nonumber \\
&= \frac{1}{4\pi} \left( \int_{\R^3} |\hat{\rho}(\omega)|^2
  \widehat{V}(\omega) \, d\omega - \int_{\R^3} |\hat{\rho}(\omega)|^2
  \widehat{V}(\omega) \hphi_{\epsilon}(\omega) \, d\omega \right), \nonumber \\
&= \frac{(2\pi)^3}{4\pi} \Bigg( \sum_{k \neq l} z_k z_l \Big( V(r_k-r_l) -
  V \ast \phi_{\epsilon} (r_k-r_l) \Big) \Bigg). \label{eqn: first
  space estimate}
\end{align}

\Cref{lem: V approx by phi} below proves that for each $0 < \beta
< 1$ there exists a constant $C$ such that:
\begin{equation} \label{eqn: second space estimate}
\forall \, k \neq l \text{ and } \forall \, \epsilon > 0, \quad |
V(r_k - r_l) - V \ast \phi_{\epsilon} (r_k - r_l) | \leq C \cdot \epsilon^{\beta}.
\end{equation}
Plugging \eqref{eqn: second space estimate} into \eqref{eqn: first
  space estimate} proves the result, since:
\begin{align*}
\frac{(2\pi)^3}{4\pi} \Bigg( \sum_{k \neq l} z_k z_l &\Big( V(r_k-r_l) -
  V \ast \phi_{\epsilon} (r_k-r_l) \Big) \Bigg) \leq \\
&\leq \frac{(2\pi)^3}{4\pi}  \sum_{k \neq l} z_k z_l | V(r_k - r_l) - V \ast
  \phi_{\epsilon} (r_k - r_l)|, \\
&\leq C \epsilon^{\beta} \sum_{k \neq l} z_k z_l, \\
&\leq C \cdot \| \rho \|_1^2 \cdot \epsilon^{\beta}.
\end{align*} 
\end{proof}

\begin{lemma} \label{lem: V approx by phi}
For each $0 < \beta < 1$ and $\delta > 0$, there exists a constant $C$
such that for any $\epsilon > 0$,
\begin{equation*}
\forall \, |u| \geq \delta, \quad | V(u) - V \ast \phi_{\epsilon} (u)
| \leq C \cdot \epsilon^{\beta}, \quad V(u) = |u|^{-1}.
\end{equation*}
\end{lemma}

\begin{proof}
Fix $\beta $ and $\delta$, and let $n$ be an arbitrary positive
integer that we will select later. Using $\phi_{\epsilon} \geq 0$ and
$\int \phi_{\epsilon} = 1$, we compute:
\begin{align*}
| V(u) &- V \ast \phi_{\epsilon}(u) | = \left| \int_{\R^3} \phi_{\epsilon}
(v) (V(u) - V(u-v) ) \, dv \right|, \\
&\leq \underbrace{\int_{|v| < \gamma} \phi_{\epsilon}(v) |V(u) - V(u-v)| \, dv}_{\text{I}} +
  \underbrace{\int_{|v| > \gamma} \phi_{\epsilon} (v) |V(u) - V(u-v)| \, dv}_{\text{II}},
\end{align*}
where $\gamma > 0$ is a free parameter that we will set as $\gamma =
\epsilon^{\beta}$ at the end of the proof. 

To bound term I, we compute the modulus of continuity of $V$ on $|u|
\geq \delta$:
\begin{equation*}
\forall \, |u|, |v| \geq \delta, \quad | V(u) - V(v) | = \frac{| |v| -
  |u| |}{|u||v|} \leq \delta^{-2} |u-v|.
\end{equation*}
Thus:
\begin{equation*}
\text{I} = \int_{|v| < \gamma} \phi_{\epsilon}(v) |V(u) - V(u-v)| \, dv \leq
\delta^{-2} \cdot \gamma \cdot \int_{|v| < \gamma} \phi_{\epsilon} (v)
\, dv \leq \delta^{-2} \cdot \gamma.
\end{equation*}

Term II is more delicate and will require several steps. We begin
with the following upper bound:
\begin{align*}
\text{II} &= \int_{|v| > \gamma} \phi_{\epsilon} (v) |V(u) - V(u-v)|
  \, dv, \\
&\leq \underbrace{\int_{|u-v| > \gamma} \phi_{\epsilon}(u-v) V(u) \, dv}_{\text{II.A}} +
  \underbrace{\int_{|u-v| > \gamma}  \phi_{\epsilon} (u-v) V(v) \, dv}_{\text{II.B}}.
\end{align*}

The first term II.A is bounded by the following calculation:
\begin{align*}
\text{II.A} = \int_{|u-v| > \gamma} \phi_{\epsilon}(u-v) V(u) \, dv
  &\leq \delta^{-1} \int_{|v| > \gamma} \phi_{\epsilon} (v) \, dv,
  \\
&\leq C_n \cdot \delta^{-1} \cdot \epsilon^{n-3} \cdot \int_{|v| > \gamma} |v|^{-n} \, dv,
  \\
&\leq C(n, \delta) \cdot \epsilon^{n-3} \cdot \gamma^{3-n}, \quad \text{if } n > 3.
\end{align*}

Term II.B is split as:
\begin{align*}
\text{II.B} &= \int_{|u-v| > \gamma}  \phi_{\epsilon} (u-v) V(v) \,
dv, \\
&= \underbrace{\int_{\substack{|u-v| > \gamma\\|v| < 1}} \phi_{\epsilon}(u-v)
V(v) \, dv}_{\text{II.B.i}} + \underbrace{\int_{\substack{|u-v| > \gamma\\|v| > 1}} \phi_{\epsilon}(u-v)
V(v) \, dv}_{\text{II.B.ii}}.
\end{align*}
We bound II.B.i using H\"{o}lder's inequality followed by an argument
similar to the one used for II.A:
\begin{align*}
\text{II.B.i} &= \int_{\substack{|u-v| > \gamma\\|v| < 1}} \phi_{\epsilon}(u-v)
V(v) \, dv, \\
&\leq \left( \int_{|u-v| > \gamma} \phi_{\epsilon}(u-v)^2 \, dv
  \right)^{\frac{1}{2}} \left( \int_{|v| < 1} V(v)^2 \, dv
  \right)^{\frac{1}{2}}, \\
&\leq C(n) \cdot \epsilon^{n-3} \cdot \gamma^{3/2 - n}.
\end{align*}
Term II.B.ii is bounded with the same argument as II.A:
\begin{equation*}
\text{II.B.ii} = \int_{\substack{|u-v| > \gamma\\|v| > 1}} \phi_{\epsilon}(u-v)
V(v) \, dv \leq \int_{|v| > \gamma} \phi_{\epsilon} (v) \, dv \leq C(n)
\cdot \epsilon^{n-3} \cdot \gamma^{3-n}.
\end{equation*}

In summary,
\begin{align}
\forall \, |u| \geq \delta, \quad |V(u) &- V \ast \phi_{\epsilon}(u)|
  \leq \text{I} + \text{II.A} + \text{II.B.i} + \text{II.B.ii},
    \nonumber \\
&\leq C(n, \delta) \cdot  (\gamma + \epsilon^{n-3} \cdot
  \gamma^{3-n} + \epsilon^{n-3} \cdot \gamma^{3/2 - n} ). \label{eqn:
  summary bound 1}
\end{align}
Recall that $\gamma$ and $n$ are free parameters. Setting $\gamma =
\epsilon^{\beta}$ and $n = (6-\beta)/2(1-\beta)$ completes the proof.
\end{proof}

\subsection{Proof of \cref{lem: fourier riemann sum}} \label{sec:
  proof of fourier riemann sum}

Recall \cref{lem: fourier riemann sum} from \cref{CoulombFourier}.

\begin{lemma}[Restatement of \cref{lem: fourier riemann sum}]
If $\rho \in \Lu (\R^3)$ decays exponentially fast as $|u| \rightarrow
\infty$, then:
\begin{equation} \label{eqn: fourier proof sketch3 v2}
\int_{\epsilon}^{\epsilon^{-1}} \| \hrho_{\alpha} \|_2^2 \,
  d\alpha =  \frac{\epsilon}{2} \left( \| \hat{\rho}_{\epsilon} \|_2^2 +
\sum_{k=2}^{\epsilon^{-2}-1}  \| \hrho_{k\epsilon} \|_2^2 + \|
\hat{\rho}_{\epsilon^{-1}} \|_2^2 \right) + O(\epsilon).
\end{equation}
\end{lemma}

\begin{proof}
The proof is an application of the trapezoid rule from numerical
integration. Recall that it approximates the integral $\int_a^b
g(\alpha) \, d\alpha$ as:
\begin{equation} \label{eqn: trapezoid rule}
\int_a^b g(\alpha) \, d\alpha = \frac{b-a}{2m} \left( g(a) +
  \sum_{k=2}^{m} 2 g(\alpha_k) + g(b) \right) -
\underbrace{\frac{(b-a)^3}{12m^2} g^{(2)}(\xi)}_{\text{error term}},
\end{equation}
where $\alpha_k = a + (k-1)(b-a)/m$ and $\xi \in [a,b]$. Take $a =
\epsilon$, $b = \epsilon^{-1}$, and $g(\alpha) = \|
\hat{\rho}_{\alpha} \|_2^2$. Since $\rho$ has exponential decay,
$\hat{\rho} \in C^{\infty}(\R^3)$ with bounded derivatives of all
orders. Therefore the error term can be bounded as:
\begin{equation*}
|\text{error term}| \leq C \frac{(\epsilon^{-1} - \epsilon)^3}{m^2} =
C \frac{\epsilon^3 (\epsilon^{-2} - 1)^3}{m^2}.
\end{equation*}
Thus if $m = \epsilon^{-2} - 1$, the resulting error term is $O(
\epsilon)$ and \eqref{eqn: fourier proof sketch3 v2} follows from
\eqref{eqn: trapezoid rule}. 
\end{proof}

Since $\hat{\rho}$ has bounded derivatives of all orders, \cref{lem:
  fourier riemann sum} can be refined to show the Coulomb
energy $U (\rho)$ can be regressed to accuracy $O(\epsilon)$ with
$O(\epsilon^{-1 - \frac{2}{n+1}})$ Fourier coefficients, for any
integer $n \geq 0$. The parameter $n$ corresponds to the order of the Newton-Cotes
numerical integration scheme \cite{hoffman:numericalMethods2001},
where for example $n=0$ and $n=1$ correspond to the rectangle and
trapezoid rules, respectively. We omit the details. 

\subsection{Proof of \cref{lem: coulomb frequency
    cut2}} \label{sec: proof of wavelet cut}

\Cref{lem: coulomb frequency cut2} is proven as a corollary
to \cref{lem: coulomb frequency cut}. 

\begin{lemma}[Restatement of \cref{lem: coulomb frequency cut2}]
Suppose that $\psi$ satisfies:
\begin{equation}
\sum_{j \in \Z} 2^{2j} \int_{[0,\pi]^2} | \hpsi (2^j r_{\theta}^{-1} \omega) |^2 \,
d\theta = |\omega|^{-2}, \quad \omega \neq 0, \label{eqn:
  littlewood-paley exact repeat}
\end{equation}
and
\begin{equation}
\supp \int_{[0,\pi]^2} | \hpsi (2^j r_{\theta}^{-1} \omega) |^2 \, d\theta \subset
\{ c_1 2^{-j} < |\omega| < c_2 2^{-j + 1} \}, \label{eqn: compact
support assumption}
\end{equation}
where $0 < c_1 \leq c_2 < \infty$ are universal constants. If $\rho$
is the sum of a Lipschitz function in $\Lu(\R^3)$ and
a finite sum of point charges, then there exists a constant $C$ such
that for any $0 < \epsilon < 1$,
\begin{equation*}
\sum_{j=-\infty}^{2\log_2 \epsilon} 2^{2j} \| \rho \ast \psi_{j, \cdot}
\|_2^2 \leq C \cdot \| \rho \|_1^2 \cdot \epsilon
\end{equation*}
and
\begin{equation*}
\sum_{j=-\log_2 \epsilon}^{+\infty} 2^{2j} \| \rho \ast \psi_{j,
  \cdot} \|_2^2 \leq C \cdot \| \rho \|_1^2 \cdot \epsilon.
\end{equation*}
\end{lemma}

\begin{proof}
Following the proof of \cref{thm: bandlimited laurent wavelet
  coulomb regression}, 
\begin{equation}
\sum_{j=-\infty}^{2\log_2 \epsilon} 2^{2j} \| \rho \ast \psi_{j, \cdot}
\|_2^2 = \int_{\R^3} | \hat{\rho} (\omega) |^2 \sum_{j = -\infty}^{2
         \log_2 \epsilon} 2^{2j} \int_{[0,\pi]^2} | \hpsi_{j, \theta}
         (\omega)|^2 \, d\theta \, d\omega, \label{eqn: wavelet
         coulomb cut 01}
\end{equation}
and
\begin{equation}
\sum_{j=-\log_2 \epsilon}^{+\infty} 2^{2j} \| \rho \ast \psi_{j, \cdot}
\|_2^2 = \int_{\R^3} | \hat{\rho} (\omega) |^2 \sum_{j = -\log_2
         \epsilon}^{+\infty} 2^{2j} \int_{[0,\pi]^2} | \hpsi_{j,
         \theta} (\omega)|^2 \, d\theta \, d\omega. \label{eqn:
         wavelet coulomb cut 02}
\end{equation}
Utilizing \eqref{eqn: compact support assumption},
\begin{equation}
\supp \sum_{j = -\infty}^{2 \log_2 \epsilon} 2^{2j} \int_{[0,\pi]^2} |
\hpsi_{j, \theta} (\omega)|^2 \, d\theta \subset \{ | \omega | > c_1
\epsilon^{-2} \}, \label{eqn: wavelet high frequency support}
\end{equation}
and
\begin{equation}
\supp  \sum_{j = -\log_2 \epsilon}^{+\infty} 2^{2j} \int_{[0,\pi]^2} |
\hpsi_{j, \theta} (\omega)|^2 \, d\theta \subset \{ | \omega | < 2
c_2 \epsilon \}. \label{eqn: wavelet low frequency support}
\end{equation}
Combing \eqref{eqn: littlewood-paley exact repeat} with \eqref{eqn:
  wavelet high frequency support} and \eqref{eqn: wavelet low
  frequency support}, and plugging into \eqref{eqn: wavelet coulomb
  cut 01} and \eqref{eqn: wavelet coulomb cut 02} respectively, yields:
\begin{equation}
\sum_{j=-\infty}^{2\log_2 \epsilon} 2^{2j} \| \rho \ast \psi_{j, \cdot}
\|_2^2 \leq \int_{|\omega| > c_1 \epsilon^{-2}}
         \frac{|\hat{\rho}(\omega)|^2}{|\omega|^2} \, d\omega \leq C
         \cdot \| \rho \|_1^2 \cdot \epsilon, \label{eqn: wavelet tail bound 01}
\end{equation}
and
\begin{equation}
\sum_{j=-\log_2 \epsilon}^{+\infty} 2^{2j} \| \rho \ast \psi_{j, \cdot}
\|_2^2 \leq \int_{|\omega| < 2 c_2 \epsilon}
         \frac{|\hat{\rho}(\omega)|^2}{|\omega|^2} \, d\omega \leq C
         \cdot \| \rho \|_1^2 \cdot \epsilon, \label{eqn: wavelet tail bound 02}
\end{equation}
where the bound \eqref{eqn: wavelet tail bound 01} follows from
\cref{lem: coulomb frequency cut}, equation \eqref{cutint2} with
$\beta = 1/2$, and the bound \eqref{eqn: wavelet tail bound 02} follows from
\cref{lem: coulomb frequency cut}, equation \eqref{cutint1}. 

\end{proof}

\bibliographystyle{plain}
\bibliography{/Users/mhirn/Dropbox/Mathematics/Bibliography/MainBib}

\end{document}